\documentclass[12pt]{article}
\usepackage{amsmath}
\usepackage{latexsym}
\usepackage{amssymb}
\newtheorem{thm}{Theorem}[section]
\newtheorem{la}[thm]{Lemma}
\newtheorem{Defn}[thm]{Definition}
\newtheorem{Remark}[thm]{Remark}
\newtheorem{Note}[thm]{Note}
\newtheorem{prop}[thm]{Proposition}

\newtheorem{Example}[thm]{Example}
\newtheorem{Examples}[thm]{Examples}
\newtheorem{Problems}[thm]{Problems}

\newtheorem{Problem}[thm]{Problem}
\newtheorem{Convention}[thm]{Convention}
\newtheorem{Number}[thm]{\!\!}
\newenvironment{defn}{\begin{Defn}\rm}{\end{Defn}}

\newenvironment{example}{\begin{Example}\rm}{\end{Example}}

\newenvironment{numba}{\begin{Number}\rm}{\end{Number}}
\newenvironment{proof}{{\noindent\bf Proof.}}%
                  {\nopagebreak\hspace*{\fill}$\Box$\medskip\medskip\par}   
\newcommand{\Punkt}{\nopagebreak\hspace*{\fill}$\Box$}

\newcommand{\ve}{\varepsilon}

\newcommand{\wt}{\widetilde}

\newcommand{\impl}{\Rightarrow}

\newcommand{\mto}{\mapsto}

\newcommand{\N}{{\mathbb N}}
\newcommand{\R}{{\mathbb R}}

\newcommand{\K}{{\mathbb K}}

\newcommand{\C}{{\mathbb C}}

\newcommand{\cO}{{\cal O}}

\newcommand{\cg}{{\mathfrak g}}
\newcommand{\ch}{{\mathfrak h}}

\newcommand{\sub}{\subseteq}

\DeclareMathOperator{\im}{im}

\DeclareMathOperator{\pr}{pr}

\DeclareMathOperator{\id}{id}

\newcommand{\cA}{{\cal A}}

\DeclareMathOperator{\evol}{evol}
\DeclareMathOperator{\Evol}{Evol}

\DeclareMathOperator{\Supp}{supp}
\DeclareMathOperator{\ev}{ev}

\DeclareMathOperator{\graph}{graph}

\begin{document}
%
%
%$\;$\\[-27mm]
\begin{center}
{\Large\bf Fundamentals of submersions and immersions\\[3mm]
 between infinite-dimensional manifolds}\\[7mm]
{\bf Helge Gl\"{o}ckner}\vspace{4mm}
\end{center}
\begin{abstract}
\noindent
We define submersions $f\colon M\to N$ between manifolds modelled
on locally convex spaces. If $N$ is finite-dimensional
or a Banach manifold, then these coincide with the na\"{\i}ve notion
of a submersion. We study pre-images of submanifolds
under submersions and pre-images under mappings whose differentials have dense image.
An infinite-dimensional version of the constant rank theorem is provided.
We also construct
manifold structures on homogeneous spaces
$G/H$ of infinite-dimensional Lie groups. Some
fundamentals of immersions between infinite-dimensional manifolds
are developed as well.\vspace{4mm}
\end{abstract}
{\small{\bf Classification.}
Primary ; 46G10; %nonlin fa, diff maps
secondary
22A22; %top and Lie groupoids
22E65; %infdim Lie
22F30; % homogeneous spaces
58B56; %infdim mfd, diff questions
58C15 %impl fctn theorems
\\[2mm]
{\bf Keywords and phrases.} Submersion, immersion, subimmersion, fibre product;
infinite-dimensional manifold, infinite-dimensional Lie group,
preimage, inverse image, constant rank theorem, Frobenius theorem, foliation, homogeneous space,
surjective differential, dense image, Banach-Lie group, regular Lie group, inverse function, implicit function,
parameters}\vspace{6mm}

\noindent
{\bf\large Introduction and statement of the results}\\[4mm]
In this article, we define and study
submersions and immersions between manifolds
modelled on arbitrary locally convex spaces,
with a view towards applications
in the theory of infinite-dimensional
Lie groupoids (cf.\ \cite{SaW})
and infinite-dimensional
Lie groups.\\[2.3mm]
Recall that a vector subspace $F$ of a (real or complex)
topological vector space $E$ is called \emph{complemented
as a topological vector space} (or \emph{complemented},
in short) if there exists a vector subspace $C\sub E$ such that
the addition map
\[
F\times C\to E,\quad (x,y)\mto x+y
\]
is an isomorphism of topological vector spaces if we endow $F$ and $C$
with the topology induced by $E$ and $F\times C$ with the product topology.
Then $F$ is closed in $E$, in particular, and $C\cong E/F$ as a topological
vector space.
We say that a vector subspace $F$ of $E$ is \emph{co-Banach}
if $F$ is complemented in~$E$ and $E/F$ is a Banach space.\\[2.3mm]
Let $\K\in \{\R,\C\}$.
If $\K=\R$,
let $r\in \N\cup\{\infty\}$ (in which case $C^r_\R$ stands for $C^r$-maps over the real
ground field as in
\cite{BGN}, \cite{RES} or \cite{GaN})
or $r=\omega$
(in which case $C^\omega_\R$ stands for real analytic maps, as in \cite{RES} or \cite{GaN}).
If $\K=\C$, let $r=\omega$ and write $C^\omega_\C$
for complex analytic maps (as in \cite{BaS}, \cite{RES} or \cite{GaN}).
In either case,
let $f\colon M\to N$ be a $C^r_\K$-map between $C^r_\K$-manifolds
modelled on locally convex topological $\K$-vector spaces.
Let $E$ be the modelling space of~$M$
and $F$ be the modelling space of~$N$.\\[4mm]
{\bf Definition.}
We say that $f$ is a \emph{$C^r_\K$-submersion}
if, for each $x\in M$, there exists a chart $\phi\colon U_\phi\to V_\phi\sub E$
of~$M$ with $x\in U_\phi$
and a chart $\psi\colon U_\psi\to V_\psi\sub F$ of~$N$ with $f(x)\in U_\psi$
such that $f(U_\phi)\sub U_\psi$ and
\[
\psi\circ f\circ \phi^{-1}=\pi|_{V_\phi}
\]
for a continuous linear map $\pi\colon E\to F$ which has a continuous
linear right inverse $\sigma\colon F\to E$ (i.e., $\pi\circ\sigma=\id_F$).\\[4mm]
Thus, identifying $E=\ker(\pi)\oplus \sigma(F)$ with $\ker(\pi)\times F$,
$\pi$ corresponds to the projection onto the second factor
(and $f$ locally looks like this projection).
The following condition may be easier to check.\\[4mm]
{\bf Definition.}
We say that $f$ is a \emph{na\"{\i}ve submersion}
if, for each $x\in M$, the continuous linear map $T_xf\colon T_xM\to T_{f(x)}N$
has a continuous linear right inverse.\\[4mm]
{\bf Remarks.}
\begin{itemize}
\item[(a)]
If $N$ is a Banach manifold, then $f\colon M\to N$ is a na\"{\i}ve submersion
if and only if $T_xf$ is surjective for each $x\in M$ and $\ker(T_xf)$ is co-Banach
in $T_xM$ (by the Open Mapping Theorem).
\item[(b)]
If $N$ is a finite-dimensional manifold, then
$f\colon M\to N$ is a na\"{\i}ve submersion if and only if $T_xf$ is surjective
for each $x\in M$ (as closed vector subspaces of finite codimension
are always complemented and hence co-Banach).
\item[(c)]
Every $C^r_\K$-submersion is a na\"{\i}ve submersion
because $T_xf$ corresponds to $\pi$ if
we identify $T_xM$ with $T_{\phi(x)}E\cong E$
and $T_{f(x)}N$ with $T_{\psi(f(x))}F\cong F$.\vspace{2mm}
\end{itemize}
{\bf Theorem A}
\emph{Let $f\colon M\to N$ be a $C^r_\K$-map between $C^r_\K$-manifolds
modelled on locally convex topological $\K$-vector spaces.
If $N$ is a Banach manifold and $r\geq 2$ or $N$ is finite-dimensional,
then $f$ is a $C^r_\K$-submersion if and only if $f$ is a na\"{\i}ve submersion.}\\[4mm]
Every $C^r_\K$-submersion $f\colon M\to N$
is an open map and admits local $C^r_\K$-sections around each point in its
image (Lemma~\ref{propsec}).
Hence, if $f\colon M\to N$ is a surjective $C^r_\K$-submersion
and $g\colon N\to Y$ a map to a $C^r_\K$-manifold $Y$
modelled on a locally convex topological $\K$-vector space,
then $g$ is $C^r_\K$ if and only if $g\circ f\colon M\to Y$ is $C^r_\K$
(Lemma~\ref{checkdiff}).
In particular, the $C^r_\K$-manifold
structure on~$N$ making the surjective map $f$ a $C^r_\K$-submersion
is essentially unique (see Lemma~\ref{uniquesub} for details).\\[2.3mm]
Let $M$ be a $C^r_\K$-manifold modelled on a locally convex topological
$\K$-vector space $E$ and $F\sub E$ be a closed vector subspace.
Recall that a subset $N\sub M$ is called a \emph{$C^r_\K$-submanifold}
of~$M$ modelled on~$F$ if, for each $x\in N$,
there exists a chart $\phi\colon U_\phi\to V_\phi\sub E$ of~$M$
with $x\in U_\phi$ such that
\[
\phi(U_\phi\cap N)=V_\phi\cap F.
\]
Then $N$ is a $C^r_\K$-manifold modelled on~$F$, with the maximal $C^r_\K$-atlas
containing the maps $\phi|_{U_\phi\cap N}\colon U_\phi\cap N\to V_\phi\cap N$
for all $\phi$ as just described.
If $E/F$ has finite dimension $k$,
then we say that $N$ has finite codimension in~$M$
and call $k$ its codimension.
If $F$ is complemented in~$E$ as a topological vector space, then $N$
is called a \emph{split}
$C^r_\K$-submanifold of~$M$.
If $N\sub M$ is a subset
all of whose connected components $C$ are open in~$N$
and each $C$ is a $C^r_\K$-submanifold of~$M$ modelled on some closed vector subspace $F_C\sub E$,
then $N$ is called a \emph{not necessarily pure} submanifold of~$M$
(see Section~\ref{secprel} for details).\\[2.3mm]
$C^r_\K$-submersions, defined as above,  go along well with fibre products
also in the case of manifolds modelled on locally convex spaces
(as announced in the author's
review \cite{REV} of \cite{Woc}; in the meantime,
a proof was also given in unpublished lecture notes, \cite[Proposition~C.8]{Wo2}).\\[4mm]
{\bf Theorem B.}
\emph{Let $M_1$, $M_2$ and $N$ be $C^r_\K$-manifolds modelled on
locally convex topological $\K$-vector spaces and
$f\colon M_1\to N$ as well as $g\colon M_2\to N$ be $C^r_\K$-maps.
If $f$ $($resp., $g)$ is a $C^r_\K$-submersion, then the fibre product}
\[
S:=M_1 \times_{f,g} M_2:=\{(x,y)\in M_1\times M_2\colon f(x)=g(y)\}
\]
\emph{is a $($not necessarily pure$)$ split $C^r_\K$-submanifold of $M_1\times M_2$,
and the projection $\pr_2 \colon M_1\times M_2\to M_2$
restricts to a $C^r_\K$-submersion $\pr_2|_S\colon S\to N$
$($resp., $\pr_1\colon M_1\times M_2\to M_1$ restricts to a $C^r_\K$-submersion
$\pr_1|_S\colon S\to M_1)$.}\\[4mm]
Submanifolds can be pulled back along submersions.\\[4mm]
{\bf Theorem C.}
\emph{Let $f\colon M\to N$ be a $C^r_\K$-submersion between
$C^r_\K$-manifolds modelled on locally convex topological
$\K$-vector spaces. Then $f^{-1}(S)$ is a $($not necessarily pure$)$ $C^r_\K$-submanifold
of $M$ for each $C^r_\K$-submanifold $S$ of~$N$,
and $f$ restricts to a $C^r_\K$-submersion $f^{-1}(S)\to S$.
Moreover, $T_x(f^{-1}(S))=(T_xf)^{-1}(T_{f(x)}S)$
for each $x\in f^{-1}(S)$.
If $S$ has finite codimension $k$ in~$M$,
then also $f^{-1}(S)$ has codimension $k$ in~$M$.
If $S$ is a split $C^r_\K$-submanifold in~$N$, then
$f^{-1}(S)$ is a split $C^r_\K$-submanifold in~$M$.}\\[4mm]
Taking $S$ as a singleton, we obtain a Regular Value Theorem:\\[4mm]
{\bf Theorem D.}
\emph{Let $f\colon M\to N$ be a $C^r_\K$-map between
$C^r_\K$-manifolds modelled on locally convex topological
$\K$-vector spaces. If $M$ has infinite dimension,
assume $r\geq 2$.
Let $y\in f(M)$. If $N$ is modelled on a Banach space,
$\ker T_xf$ is co-Banach in $T_xM$ and $T_xf$ is surjective
for all $x\in f^{-1}(\{y\})$, then
$f^{-1}(\{y\})$ is a $($not necessarily pure$)$ split $C^r_\K$-submanifold of~$M$
and $T_x(f^{-1}(\{y\}))=\ker T_xf$ for each $x\in f^{-1}(\{y\})$.
If $N$ has finite dimension~$k$, then $f^{-1}(\{y\})$ has codimension
$k$ in~$M$.}\\[4mm]
For finite-dimensional $N$, we recover
the Regular Value Theorem from~\cite{NaW}:\\[4mm]
\emph{Let $f\colon M\to N$ be a $C^r_\K$-map from
a $C^r_\K$-manifold $M$ to a finite-dimensional $C^r_\K$-manifold $N$.
Let $y\in f(M)$. If $T_xf$ is surjective
for all $x\in f^{-1}(\{y\})$, then
$f^{-1}(\{y\})$ is a $($not necessarily pure$)$ split $C^r_\K$-submanifold of~$M$,
of codimension $\dim_\K(N)$.}\\[4mm]
Also the following
variant may be of interest (compare \cite{Fre}
for the case of tame Fr\'{e}chet spaces):\\[4mm]
{\bf Theorem E.}
\emph{Let $M$ and $N$ be $C^r_\K$-manifolds
modelled on locally convex topological $\K$-vector spaces,
$S\sub N$ be a $C^r_\K$-submanifold of finite codimension~$k$
and
$f\colon M\to N$ be a $C^r_\K$-map such that}
\[
T_xf\colon T_xM\to T_{f(x)}N
\]
\emph{has dense image for each $x\in f^{-1}(S)$.
Then $f^{-1}(S)$ is a $($not necessarily pure$)$ split $C^r_\K$-submanifold of $M$, of finite codimension $k$.
Moreover, $T_x(f^{-1}(S))=(T_xf)^{-1}(T_{f(x)}S)$ for each $x\in f^{-1}(S)$.}\\[4mm]
If $M$ and $N$ are $C^r_\K$-manifolds modelled
on locally convex topological $\K$-vector spaces,
we call a $C^r_\K$-map $f\colon M\to N$ a \emph{densemersion}
if $T_xf\colon T_xM\to T_{f(x)}N$ has dense image for each $x\in M$.
As a special case of Theorem~E, we get:\\[4mm]
\emph{If a $C^r_\K$-map $f\colon M\to N$ is a densemersion,
then $f^{-1}(S)$ is a $($not necessarily pure$)$ $C^r_\K$-submanifold of~$M$ of codimension~$k$,
for each $C^r_\K$-submanifold~$S$ of $N$ of finite codimension~$k$.
Moreover, $T_x (f^{-1}(S))=(T_xf)^{-1}(T_{f(x)}S)$
for all $x\in f^{-1}(S)$.}\\[4mm]
We also have a version of the Constant Rank Theorem.\\[4mm]
{\bf Theorem F.}
\emph{Let $M$ be a $C^r_\K$-manifold modelled on a locally convex topological
$\K$-vector space, $N$ be a finite-dimensional $C^r_\K$-manifold, $\ell\in \N$
and $f\colon M\to N$ be a $C^r_\K$-map such that the image $\im(T_xf)$
has dimension $\ell$, for each $x\in M$.
Let $n$ be the dimension of~$N$.
Then, for each $x_0\in M$,
there exists a chart $\phi\colon U_\phi\to V_\phi\sub E$ of $M$
and a chart $\psi\colon U_\psi\to V_\psi\sub \R^n$ of~$N$
such that $x_0\in U_\phi$, $f(U_\phi)\sub U_\psi$
and}
\[
(\psi\circ f\circ \phi^{-1})(x)=(\pi(x),0)\quad\mbox{for all $\,x\in V_\phi$,}
\]
\emph{for a continuous linear map $\pi\colon E\to \R^\ell$
which has a continuous linear right inverse.
In particular, $f^{-1}(\{y\})$ is a split $C^r_\K$-submanifold
of~$M$ of codimension~$\ell$,
for each $y\in f(M)$, and $T_x f^{-1}(\{y\})=\ker T_xf$
for each $x\in f^{-1}(\{y\})$.
Moreover, each $x_0\in M$ has an open neighbourhood $O\sub M$ such that
$f(O)$ is a $C^r_\K$-submanifold of~$N$ and there exists a $C^r_\K$-diffeomorphism}
\[
\theta\colon f(O)\times S\to O
\]
\emph{such that $\theta(\{y\}\times S)=O\cap f^{-1}(\{y\})$ for all $y\in f(O)$.}\\[4mm]
Note that
an analogous local decomposition is available for $C^r_\K$-submersions
as these locally look like a projection $\pr_1\colon E_1\times E_2\to E_1$,
in which case we can take $S:=E_2$ and $\theta:=\id_{E_1\times E_2}$
(see Lemma~\ref{locdeco} for details).\\[2.3mm]
For a Constant Rank Theorem for mappings
between Banach manifolds (with or without boundary or corners)
and further results on subimmersions,
see already~\cite{MaO}. The book also
contains an analogue of Godement's theorem
for Banach manifolds~$M$, i.e., a characterization
of those equivalence relations $\sim$ on~$M$
for which the space $M/\!\!\sim$
of equivalence classes can be given
a Banach manifold structure which turns the canonical
quotient map $q\colon M\to M/\!\!\sim$, $\,x\mto [x]$
into a submersion (cf.\ \cite{Ser}
for the finite-dimensional case).\\[2.3mm]
If $G$ is a Lie group modelled
on a locally convex space and $g\in G$,
let $\lambda_g\colon G\to G$ be the left translation map $x\mto gx$
and write $g.v:=T\lambda_g(v)$ for $v\in TG$.
Let $e$ be the neutral element of~$G$ and $\cg:=L(G):=T_eG$
be its Lie algebra.
We recall that $G$ is called \emph{regular}
if the initial value problem
\[
\eta'(t)=\eta(t).\gamma(t),\quad \eta(0)=e
\]
has a (necessarily unique) solution
$\Evol(\gamma):=\eta\in C^\infty([0,1],G)$ for
each $\gamma\in C^\infty([0,1],\cg)$
and the map
\[
\evol\colon C^\infty([0,1],\cg)\to G,\quad \gamma\mto\Evol(\gamma)(1)
\]
is smooth (see \cite{GaN}, \cite{SUR}, cf.\ \cite{Mil}).\\[2.3mm]
We have results on homogeneous
spaces of infinite-dimensional Lie groups.
Notably, we obtain (as announced in \cite[\S4(b)]{OWR}:\\[4mm]
{\bf Theorem G.}
\emph{Let $G$ be a $C^r_\K$-Lie group
modelled on a locally convex space
and $H$ be a subgroup of $G$ which is a split
$C^r_\K$-submanifold of $G$
such that}
\begin{itemize}
\item[(a)]
\emph{$H$ is a regular Lie group
and $H$ is co-Banach; or}
\item[(b)] 
\emph{$H$ is a Banach-Lie group.}
\end{itemize}
\emph{Then $G/H$ can be given a $C^r_\K$-manifold structure which turns the canonical quotient
map $q\colon G\to G/H$ into a $C^r_\K$-submersion.
If $H$ is, moreover, a normal subgroup of~$G$,
then $G/H$ is a $C^r_\K$-Lie group.}\\[4mm]
Note that $G/H$ is a Banach manifold
in the situation of~(a).
We also have a more technical result
independent of regularity (Proposition~\ref{noregu}).\\[2.3mm]
Recall that a map $j\colon X\to Y$ between topological spaces
is called a \emph{topological embedding}
if $j$ is a homeomorphism onto its image, i.e,
the co-restriction $j|^{j(X)}\colon X\to j(X)$
is a homeomorphism if we endow the image $j(X)$ with the topology induced by~$Y$.\\[2.3mm]
Return to $r\in\N\cup\{\infty,\omega\}$ if $\K=\R$,
$r=\omega$ if $\K=\C$.
Let $f\colon M\to N$ be a $C^r_\K$-map between $C^r_\K$-manifolds
modelled on locally convex topological $\K$-vector spaces.
Let $E$ be the modelling space of~$M$
and $F$ be the modelling space of~$N$.\\[4mm]
{\bf Definition.}
We say that $f$ is a \emph{$C^r_\K$-immersion}
if, for each $x\in M$, there exists a chart $\phi\colon U_\phi\to V_\phi\sub E$
of~$M$ with $x\in U_\phi$
and a chart $\psi\colon U_\psi\to V_\psi\sub F$ of~$N$ with $f(x)\in U_\psi$
such that $f(U_\phi)\sub U_\psi$ and
\[
\psi\circ f\circ \phi^{-1}=j|_{V_\phi}
\]
for a linear map $j \colon E\to F$ which
is a topological embedding onto a complemented vector subspace
of~$F$. If $f$ is both a $C^r_\K$-immersion and a topological embedding,
then we call $f$ a \emph{$C^r_\K$-embedding}.\\[4mm]
We shall see that a $C^r_\K$-map $f\colon M\to N$ is a $C^r_\K$-embedding if and only
if $f(M)$ is a split $C^r_\K$-submanifold of~$N$ and $f|^{f(M)}\colon M\to f(M)$ is
a $C^r_\K$-diffeomorphism (Lemma~\ref{easieremb}).\\[2.3mm]
The following condition may be easier to check.\\[4mm]
{\bf Definition.}
We say that $f$ is a \emph{na\"{\i}ve immersion}
if, for each $x\in M$,
the tangent map $T_xf\colon T_xM\to T_{f(x)}N$ is a topological
embedding whose image is a complemented
vector subspace of $T_{f(x)}N$.\\[4mm]
{\bf Remarks.}
\begin{itemize}
\item[(a)]
If $M$ is a finite-dimensional manifold, then
$f\colon M\to N$ is a na\"{\i}ve immersion if and only if $T_xf$ is injective
for each $x\in M$.
\item[(b)]
Every $C^r_\K$-immersion is a na\"{\i}ve immersion.\vspace{2mm}
\end{itemize}
{\bf Theorem H.}
\emph{Let $f\colon M\to N$ be a $C^r_\K$-map between $C^r_\K$-manifolds
modelled on locally convex topological $\K$-vector spaces.
Assume that}
\begin{itemize}
\item[(a)]
\emph{$M$ is a Banach manifold and $r\geq 2$; or}
\item[(b)]
\emph{$M$ is finite-dimensional.}
\end{itemize}
\emph{Then $f$ is a $C^r_\K$-immersion if and only if $f$ is a na\"{\i}ve immersion.}\\[4mm]
Important examples of $C^r_\K$-maps with injective tangent
maps arise from actions of Lie groups.
Let us say that
a Lie group $G$ modelled on a locally convex space \emph{has an exponential function}
if, for each $v\in L(G)$, there is a (necessarily unique) smooth homomorphism
$\gamma_v\colon \R\to G$ such that $(\gamma_v)'(0)=v$.
In this case, we define
\[
\exp_G\colon L(G)\to G,\quad \exp_G(v)=\gamma_v(1).
\]
Then $\gamma_v(t)=\exp_G(tv)$ for all $v\in L(G)$ and $t\in \R$.\\[2.4mm]
Now consider a
smooth Lie group $G$ (modelled on a locally convex space)
which has an exponential function. Let
$\sigma\colon G\times M\to M$, $(g,x)\mto gx$ be a $C^1$-action
on a $C^1$-manifold~$M$ (modelled on some locally convex space)
and $x\in M$ with stabilizer $G_x$ such that there is a smooth manifold
structure on $G/G_x$ which turns the quotient map
$q\colon G\to G/G_x$ into a smooth submersion.
Then the orbit map
\[
b\colon G\to M,\quad g\mto gx
\]
factors to a $C^1$-map
\[
\wt{b}\colon G/G_x\to M,\quad gG_x\mto gx.
\]
We observe:\\[4mm]
{\bf Theorem I.}
\emph{For each $g\in G$, the tangent map $T_{gG_x}\wt{b}\colon T_{gG_x}(G/G_x)\to T_{gx}M$
is injective.}\\[4mm]
This result is of course known in many special cases (like finite-dimensional groups),
and the proof can follow the familiar pattern.\\[2.3mm]
To illustrate our results on immersions and embeddings, let us look at an example.
Let $M$ be a compact smooth manifold and $H$ be
a compact Lie group.
We consider
the conjugation action of $H$
(viewed as the group of constant maps) on the Fr\'{e}chet-Lie group
$G:=C^\infty(M,H)$. Let $\gamma\in G$.
Combining Theorem~H and Theorem~I,
we see that the conjugacy class
\[
\gamma^H:=\{h\gamma h^{-1}\colon h\in H\}
\]
is a split real analytic submanifold in~$G$
and $H/H_\gamma\to \gamma^H$, $hH_\gamma\mto h\gamma h^{-1}$ a real analytic diffeomorphism
(see Example~\ref{conjuex}
for details).\\[4mm]
Let $r\in \N\cup\{\infty\}$. A $C^r$-manifold $M$ modelled
on a locally convex space is called \emph{$C^r$-regular}
if the topology on~$M$ is initial with respect to the set
$C^r(M,\R)$ of all real-valued $C^r$-functions on~$M$.
It is well-known that $M$ is $C^r$-regular if and only
if, for each $x\in M$ and neighbourhood $U\sub M$ of~$x$,
there exists a $C^r$-function $f\colon M\to \R$ with image in $[0,1]$,
support $\Supp(f)\sub U$ and such that $f|_W=1$ for some neighbourhood
$W\sub U$ of~$x$ (see, e.g., \cite{GaN}).
Recall that if $M$ is a $\sigma$-compact finite-dimensional $C^r$-manifold,
of dimension~$m$, say, then $M$
is $C^r$-regular and admits a $C^r$-embedding into $\R^{2m}$,
by Whitney's Embedding Theorem.
As far as infinite-dimensional manifolds
are concerned,
embeddings of
separable metrizable Hilbert manifolds (and some
more general Banach manifolds
and Fr\'{e}chet manifolds) as open subsets
of the modelling space have been studied,
see \cite{Eel}, \cite{Hen}
and the references therein.
The following theorem generalizes a result of~\cite{Dah}
concerning embeddings of manifolds modelled on $\R^I$ for some
set~$I$.\\[4mm]
{\bf Theorem J.}
\emph{Let $r\in \N\cup\{\infty\}$ and $M$ be a $C^r$-regular
$C^r$-manifold modelled on a real locally convex space~$E$.
Then there exists a $C^r$-embedding
$\theta \colon M\to (E\times \R)^M$.}\\[4mm]
\emph{Acknowledgement.} The author thanks A. Schmeding (Trondheim)
for discussions concerning infinite-dimensional Lie groupoids,
which prompted the Constant Rank Theorem
presented in this paper. Also the question solved in Example~\ref{conjuex}
was posed by him.
\section{Preliminaries and basic facts}\label{secprel}
We write $\N=\{1,2,\ldots\}$ and $\N_0:=\N\cup\{0\}$
and put a total order on $\N\cup\{\infty,\omega\}$
by ordering $\N$ as usual and declaring $k< \infty<\omega$
for each $k\in\N$.
All topological spaces and all topological vector spaces considered
in the paper are assumed Hausdorff.
We fix a ground field $\K\in\{\R,\C\}$.
Henceforth, the phrase \emph{locally convex space}
abbreviates \emph{locally convex topological $\K$-vector space}.
Our general references for differential calculus of $C^r$-maps,
real analytic and complex analytic mappings are \cite{RES}
and \cite{GaN} (cf.\ also \cite{BGN}).
If $E$ and $F$ are locally convex spaces, $U\sub E$ an open set and
$f\colon U\to F$ a $C^1$-map, we write
\[
df(x,y):=\frac{d}{dt}\Big|_{t=0}f(x+ty)
\]
for the directional derivative of $f$ at $x\in U$ in the direction
$y\in E$. Then $f'(x):=df(x,.)\colon E\to F$
is a continuous linear map (see \cite{RES} or \cite{GaN}).
Having fixed $\K$, let $r$ be as in the Introduction.
Recall that a $C^r_\K$-manifold modelled on a locally convex space~$E$
is a Hausdorff topological space $M$,
together with a maximal $C^r_\K$-atlas $\cA$,
consisting of homeomorphism $\phi\colon U_\phi\to V_\phi$ from
open subsets $U_\phi\sub M$ onto open sets $V_\phi\sub E$
(the \emph{charts})
such that $\bigcup_{\phi\in\cA}U_\phi=M$
and $\phi\circ \psi^{-1}\colon \psi(U_\phi\cap U_\psi)\to E$
is $C^r_\K$ for all $\phi,\psi\in \cA$.
Manifolds of this form are also called \emph{pure manifolds}.
A topological space $M$ with open connected
components,
together with a $C^r_\K$-manifold structure modelled on  locally convex space $F_C$
on each connected component~$C$ of~$M$,
is called a (not necessarily pure) $C^r_\K$-manifold.
Unless we explicitly state the contrary, only pure manifolds
are considered in this article.
The definition of a $C^r_\K$-submanifold $N\sub M$ modelled on a closed vector subspace
$F$ of the modelling space $E$ of~$M$
was already given in the Introduction.
If $N$ is such a submanifold and, moreover,
$F$ is complemented in~$E$, then we call $N$
a \emph{split} $C^r_\K$-submanifold of~$M$.
A subset $N$ of a $C^r_\K$-manifold $M$ is called a (not necessarily pure)
$C^r_\K$-submanifold of~$M$ if $N$ has open connected components
in the induced topology and each connected component $C$ of~$N$
is a submanifold of $M$ modelled on some closed vector subspace $F_C$
of the modelling space~$E$ of~$M$ in the previous sense (as described in the Introduction).
If, moreover, each $F_C$ is complemeted in~$E$,
then $N$ is called a \emph{split} (not necessarily pure) submanifold
of~$M$.
For example, if
$f\colon M\to N$ is a $C^r_\K$-map between $C^r_\K$-manifolds,
then its graph
\[
\mbox{graph}(f)=\{(x,f(x))\colon x\in M\}
\]
is a split $C^r_\K$-submanifold of $M\times N$.
The following fact is used repeatedly (without further mention):
If $f\colon M\to N$ is a map between $C^r_\K$-manifolds
and $f(M)\sub S$ for some $C^r_\K$-submanifold $S\sub N$,
then $f$ is $C^r_\K$ if and only if $f|^S\colon M\to S$ is $C^r_\K$ (see \cite{GaN}).
In this article, the words \emph{$C^r_\K$manifold} (and \emph{$C^r_\K$-Lie group})
always refer to
$C^r_\K$-manifolds (and $C^r_\K$-Lie groups) modelled
on locally convex spaces (which need not be finite-dimensional).
\begin{numba}\label{invpara}
Fix $\K$ and let $r\in \N\cup\{\infty,\omega\}$ be as in the Introduction.
Our main tool is the \emph{Inverse Function Theorem with Parameters}
from \cite{IMP} (for smooth or $\K$-analytic $f$ or finite-dimensional $F$)
and its strengthened version \cite{IM2} (including the $C^r$-case for finite $r$,
without loss of derivatives):\\[4mm]
\emph{Let $E$ be a locally convex space, $F$ be a Banach space,
$U\sub E$ and $V\sub F$ be open, $(x_0,y_0)\in U\times V$
aand $f\colon U\times V\to F$ be a $C^r_\K$-map
such that $f_{x_0}\colon V\to F$, $y\mto f(x_0,y)$
has invertible differential $f_{x_0}'(y_0)\colon F\to F$
at~$y_0$. If $F$ does not have finite dimension, assume that $r\geq 2$.
After shrinking the open neighbourhoods $U$ of $x_0$
and $V$ of~$y_0$, one can achieve the following}:
\begin{itemize}
\item[(a)]
\emph{$f_x:=f(x,.)\colon V\to F$ has open image and $f_x\colon V\to f_x(V)$ is a $C^r_\K$-diffeomorphism,
for each $x\in U$; and}
\item[(b)]
\emph{The image $\Omega$ of the map $\theta\colon U\times V\to U\times F$, $\theta(x,y):=(x,f_p(y))$
is open in $U\times F$ and $\theta\colon U\times V\to\Omega$ is a $C^r_\K$-diffeomorphism
with inverse}
\[
\theta^{-1}\colon\Omega\to U\times V,\quad (x,z)\mto(x,f_x^{-1}(z)).
\]
\emph{In particular, the map $\Omega\to V$, $(x,z)\mto (f_x)^{-1}(z)$ is $C^r_\K$.}
\end{itemize}
\end{numba}
In the remainder of this section, we compile some simple
basic facts on $C^r_\K$-submersions
and submanifolds. Certainly several of them may be known or
part of the folklore, but we find it useful to assemble them at one place
in self-contained form.
\begin{la}\label{locapro}
A $C^r_\K$-map $f\colon M\to N$ between $C^r_\K$-manifolds is a $C^r_\K$-submersion
if and only if, for each $x\in M$, there exists an open neighbourhood $U\sub M$ of $x$,
an open neighbourhood $U_1\sub N$ of $f(x)$, a $C^r_\K$-manifold $U_2$
and a $C^r_\K$-diffeomorphism $\theta\colon U_1\times U_2\to U$ such that
$f(U)=U_1$ and $f\circ \theta=\pr_1\colon U_1\times U_2\to U_1$, $(x_1,x_2)\mto x_1$.
\end{la}
\begin{proof}
If $f$ is a $C^r_\K$-submersion and $x\in M$, let $\phi\colon U_\phi\to E$
and $\psi\colon U_\psi\to V_\psi\sub F$ be charts of $M$ and $N$, respectively, such that
$x\in U_\phi$, $f(U_\phi)\sub U_\psi$ and $\psi\circ\phi^{-1}=\pi|_{V_\phi}$
for a continuous linear map $\pi\colon E\to F$ admitting a continuous linear right
inverse $\sigma\colon F\to E$. Then $\alpha:=\pi|_{\sigma(F)}\colon \sigma(F)\to F$
is an isomorphism of topological vector spaces.
After replacing the modelling space $F$ of $N$ with $\sigma(F)$ and $\psi$ with $\alpha^{-1}\circ\psi$,
we may assume that $E=F\oplus \ker(\pi)$ as a topological vector space and
$\pi$ is the projection $F\oplus \ker(\pi)\to F$, $(x_1,x_2)\mto x_1$.
After shrinkung $U_\phi$, we may assume that $V_\phi=P\times Q$
with open subsets $P\sub F$ and $Q\sub\ker(\pi)$.
We may now replace $V_\psi$ with $P$ and hence assume $V_\phi=P$.
It remains to set $U:=U_\phi$, $U_1:=U_\psi$ and $U_2:=Q$.
Let $\pr_j\colon U_1\times U_2\to U_1$, $(x_1,x_2)\mto x_j$ be the projection
onto the $j$-th component for $j\in\{1,2\}$.
Then $\theta:=\phi^{-1}\circ (\psi\times \id_Q)\colon U_1\times Q\to U$
is a $C^r_\K$-diffeomorphism such that $f\circ \theta=\psi^{-1}\circ\psi\circ f\circ \phi^{-1}\circ (\psi\times \id_Q)
=\psi^{-1}\circ\psi\circ f\circ \phi^{-1}\circ (\psi\circ\pr_1,\pr_2)
=\psi^{-1}\circ \pi|_{P\times Q}\circ(\psi\circ \pr_1,\pr_2)=\psi^{-1}\circ\psi\circ\pr_1=\pr_1$.

If, conversely, for $x\in M$ we can find $U$, $U_1$, $U_2$ and $\theta$ as described in the lemma,
then we choose charts $\phi\colon U_\phi\to V_\phi\sub E_1$
and $\psi\colon U_\psi\to V_\psi\sub E_2$ for $U_1$ and $U_2$, respectively, such that
$\theta^{-1}(x)\in U_\phi\times U_\psi$. Set $W:=\theta(U_\phi\times U_\psi)$.
Let $\pr_1\colon U_1\times U_2\to U_1$ and $\pi\colon E_1\times E_2\to E_1$
be the projection onto the first component.
Then $\tau:=(\phi\times \psi)\circ \theta^{-1}|_W\colon W\to V_\phi\times V_\psi\sub E_1\times E_2$
is a $C^r_\K$-diffeomorphism and $\phi\circ f\circ\tau^{-1}=\phi\circ f\circ \theta\circ (\phi^{-1}\times\psi^{-1})
=\phi\circ \pr_1\circ (\phi^{-1}\times \psi^{-1})=\pi|_{V_\phi}$, entailing that $f$ is a $C^r_\K$-submersion. 
\end{proof}
\begin{la}\label{composub}
Let $f\colon M\to N$ and $g\colon N\to K$ be $C^r_\K$-submersions between
$C^r_\K$-manifolds. Then also $g\circ f\colon M\to K$ is a $C^r_\K$-submersion.
\end{la}
\begin{proof}
Given $x\in M$, we find an open neighbourhood $U\sub M$ of $x$
and a $C^r_\K$-diffeomorphism $\theta\colon U_1\times U_2\to U$
for some open neighbourhood $U_1$ of $f(x)$ in~$N$ and some $C^r_\K$-manifold $U_2$,
such that $f\circ\theta=\pr_1\colon U_1\times U_2\to U_1$, $(x_1,x_2)\mto x_1$
(by Lemma~\ref{locapro}).
Applying now Lemma~\ref{locapro} to $g$, we see that,
after shrinking $U_1$, we may assume that there exists
an open neighbourhood $V_1\sub K$ of $g(f(x))$ and a $C^r_\K$-diffeomorphism
$\zeta\colon V_1\times V_2\to U_1$
such that $g\circ \zeta=\pi\colon V_1\times V_2\to V_1$, $(x_1,x_2)\mto x_1$.
Consider the projections $p\colon V_1\times V_2\times U_3\to V_1\times V_2$, $(x_1,x_2,x_3)\mto (x_1,x_2)$
and $q\colon V_1\times V_2\times U_3\to V_1$, $(x_1,x_2,x_3)\mto x_1$.
Then $\eta:=\theta\circ (\zeta\times \id_{U_2})\colon V_1\times V_2\times U_2\to U$
is a $C^r_\K$-diffeomorphism such that $g\circ f\circ \eta=g\circ f\circ \theta\circ (\zeta\times \id_{U_2})
=g\circ\pr_1\circ (\zeta\times \id_{U_2})=g\circ \zeta\circ p=\pi\circ p=q$.
Hence $g\circ f$ is a $C^r_\K$-submersion, by Lemma~\ref{locapro}.
\end{proof}
\begin{la}\label{subinsub}
Let $M$ be a $C^r_\K$-manifold,
$N$ be a split $C^r_\K$-submanifold of $M$ and
$K$ be a $C^r_\K$-submanifold $($resp., a split $C^r_\K$-submanifold$)$ of~$N$.
Then $K$ is a $C^r_\K$-submanifold $($resp., a split $C^r_\K$-submanifold$)$ of~$M$.
\end{la}
\begin{proof}
Let $E$ be the modelling space of~$M$, $F\sub E$ be the modelling space of~$N$ and $Y\sub F$
be the modelling space of~$K$. By hypothesis, $E=F\oplus C$
as a topological vector space for some vector subspace $C\sub E$.
Given $x\in K$, let $\phi\colon U_\phi\to V_\phi\sub E$
be a chart of~$M$ such that $x\in U_\phi$ and $\phi(N\cap U_\phi)=F\cap V_\phi$.
After shrinking $V_\phi$, we may assume that $V_\phi=P\times Q$ with open subsets $P\sub F$ and $Q\sub C$.
Next, we choose a chart $\psi\colon U_\psi\to V_\psi\sub F$ of~$N$
such that $\psi(K\cap U_\psi)=Y\cap V_\psi$.
After shrinking $U_\psi$ and $P$ if necessary, we may assume that $U\cap N=\phi^{-1}(P)=U_\psi$.
Then $\theta:=((\psi\circ \phi^{-1}|_P)\times \id_Q)\circ \phi$ is a $C^r_\K$-diffeomorphism
$U\to P\times Q\sub E$ and hence a chart of $M$.
If $x\in U\cap K$, then $x\in U\cap N$ in particular
and hence $\phi(x)\in P$, entailing that $\theta(x)=\psi(x)$.
Now $x\in U\cap K=(U\cap N)\cap K=U_\psi\cap K$ entails $\theta(x)=\psi(x)\in Y$.
Hence $\theta(U\cap K)\sub (P\times Q)\cap Y$.
If conversely, $x\in U$ such that $\theta(x)\in (P\times Q)\cap Y$,
then $\theta(x)\in P\cap Y\sub P$ and thus $\phi(x)\in P$, entailing that $x\in N$.
Hence $\theta(x)=\psi(x)$ and as this is in~$Y$, we have $x\in K$.
Thus $\theta(K\cap U)=Y\cap (P\times Q)$,
showing that $K$ is a $C^r_\K$-submanifold of~$M$.
If $K$ is a split submanifold of $N$, then $F=Y\oplus D$
as a topological vector space for some vector subspace $D\sub F$
and hence $E=F\oplus C= Y\oplus (D\oplus C)$.
Thus $Y$ is complemented in~$E$ and so $K$ is a split $C^r_\K$-submanifold of~$M$.
\end{proof}
\begin{la}\label{pureatx}
Let $M$ be a $C^r_\K$-manifold modelled on a locally convex space~$E$
and $N\sub M$ be a subset.
Assume that, for each $x\in N$,
there exists a closed vector subspace $F_x\sub E$
and a $C^r_\K$-diffeomorphism $\phi_x\colon U_x\to V_x$
from an open neighbourhood $U_x$ of~$x$ in~$M$ onto an open subset $V_x\sub E$
such that $\phi_x(N\cap U_x)=F_x\cap V_x$.
Then $N$ is a $($not necessarily pure$)$ $C^r_\K$-submanifold of~$M$.
If each $F_x$ can be chosen as a complemented vector subspace of $E$,
then $N$ is a $($not necessarily pure$)$ split $C^r_\K$-submanifold of~$M$. 
\end{la}
\begin{proof}
If $x,y\in N$, write $x\sim y$ if there exist $n\in \N$ and $x_1,\ldots, x_n\in N$
such that $x\in U_{x_1}$, $y\in U_{x_n}$
and $U_{x_j}\cap U_{x_{j+1}}\not=\emptyset$
for all $j\in \{1,\ldots, n-1\}$.
It is clear that $\sim$ is an equivalence relation on~$N$.
Write $[x]$ for the equivalence class of $x\in N$.
Since $y\sim z$ for all $z\in U_y$, we have $U_y\sub [x]$ for all $y\in [x]$
and hence $[x]$ is open.
Hence $[x]$ is open and closed, entailing that $[x]$ contains the connected component
$N_y$ of $N$ containing~$y$ for each $y\in [x]$.
As the open neighbourhood $U_x\cap N$ of $x\in N$ in~$N$ is homeomorphic to the open set $V_x\cap F_x$ in the locally
convex space $F_x$, we see that all connected components of~$N$ are open in~$N$.
We claim that, if $y\in [x]$,
then there exists an automorphism $\alpha\colon E\to E$
of the topological vector space~$E$ such that $\alpha(F_x)=F_y$.
We may therefore replace $F_y$ with $F_x$ for each $y\in [x]$
and conclude that $N$ is a (not necessarily pure) $C^r_\K$-submanifold of~$M$.
Since $F_y=\alpha(F_x)$ is complemented in~$E$ if $F_x$ is complemented in~$E$,
we see that $N$ is a (not necessarily pure) split $C^r_\K$-submanifold of~$M$
whenever $F_x$ can be chosen complemented in~$E$ for each $x\in N$.
\end{proof}
\begin{la}
If $f_j\colon M_j\to N_j$ is a $C^r_\K$-submersion between $C^r_\K$-manifolds
for~$j\in \{1,2\}$, then also
$f_1\times f_2\colon M_1\times M_2\to N_1\times N_2$ is a $C^r_\K$-submersion.
\end{la}
\begin{proof}
If $x=(x_1,x_2)\in M_1\times M_2$, we find charts $\phi_j\colon U_j\to V_j\sub E_j$ for $M_j$ with $x_j\in U_j$
and charts $\psi_j\colon P_j\to Q_j\sub F_j$ for $N_j$ such that $f_j(U_j)\sub P_j$
and $\psi_j\circ f_j\circ \phi_j^{-1}=\pi_j|_{V_j}$ for a continuous linear map $\pi_j\colon E_j\to F_j$
admitting a continuous linear right inverse $\sigma_j\colon F_j\to E_j$
for $j\in \{1,2\}$.
Then $\pi_1\times \pi_2$ is continuous linear with $\sigma_1\times \sigma_2$
as a continuous linear right inverse; moreover,
$\phi_1\times \phi_2$ is a chart for $M_1\times M_2$ with $x=(x_1,x_2)$
in its domain and $(\psi_1\times \psi_2)\circ (f_1\times f_2)\circ (\phi_1\times \phi_2)^{-1}=(\pi_1\times \pi_2)|_{U_1\times U_2}$.
Thus $f_1\times f_2$ is a $C^r_\K$-submersion.
\end{proof}
\begin{la}\label{propsec}
If $f\colon M\to N$ is a $C^r_\K$-submersion between $C^r_\K$-manifolds,
then $f$ is an open map. Moreover, for each $x\in M$
there exists an open neighbourhood $U\sub N$ and a $C^r_\K$-map $\tau\colon U\to M$
such that $f\circ \tau=\id_U$.
\end{la}
\begin{proof}
For each $x\in M$, we find a chart $\phi\colon U_\phi\to V_\phi\sub E$ of $M$ with $x\in U_\phi$
and a chart $\psi\colon U_\psi\to V_\psi\sub F$ of~$N$
such that $f(U_\phi)\sub U_\psi$ and $\psi\circ f\circ \phi^{-1}=\pi|_{V_\phi}$
for a continuous linear map $\pi\colon E\to F$ admitting a continuous linear right inverse
$\sigma\colon F\to E$.
After replacing $\phi$ with $\phi-\phi(x)$
and $\psi$ with $\psi-\pi(\phi(x))$,
we may assume that $\phi(x)=0$ and $0=\pi(\phi(x))=\psi(f(x))\in V_\psi$
As a consequence, there exists a $0$-neighbourhood $P\sub V_\psi$ such that
$\sigma(P)\sub V_\phi$. Then $\tau\colon \psi^{-1}(P)\to M$, $\tau:=\phi^{-1}\circ \sigma\circ \psi|_{\psi^{-1}(P)}$
is a $C^r_\K$-map such that $\tau(f(x))=\phi^{-1}(\sigma(0)))=\phi^{-1}(0)=x$
and $f\circ \tau=f\circ \phi^{-1}\circ \sigma\circ \psi|_{\psi^{-1}(P)}=\psi^{-1}\circ\psi\circ f\circ \phi^{-1}\circ \sigma\circ
\psi|_{\psi^{-1}(P)}=\psi^{-1}\circ \pi\circ \sigma\circ \psi|_{\psi^{-1}(P)}=\id_{\psi^{-1}(P)}$.
Thus $f\colon M\to N$ admits local $C^r_\K$-sections. If $W\sub M$ is open and
$\sigma_x\colon Q_x\to M$ for $x\in W$ a $C^r_\K$-map on an open set $Q_x\sub N$ with $f(x)\in Q_x$
and $f\circ \sigma_x=\id_{Q_x}$,
then $\sigma_x^{-1}(W)$ is an open neighbourhood of $f(x)$ in $Q_x$ (hence in $N$)
and $f(W)\supseteq f(\sigma_x(\sigma_x^{-1}(W)))=\sigma_x^{-1}(W)$.
This $f(W)$ is a neighbourhood of each point $f(x)$ in $f(W)$ in~$N$
and hence $f(W)$ is open.
\end{proof}
\begin{la}\label{checkdiff}
Let $q\colon M\to N$ be a surjective $C^r_\K$-submersion between
$C^r_\K$-manifolds and $f\colon N\to K$ be a map to a $C^r_\K$-manifold $K$.
Then $f$ is $C^r_\K$ if and only if $f\circ q\colon M\to K$ is $C^r_\K$.
\end{la}
\begin{proof}
If $f$ is $C^r_\K$, then also $f\circ q$ is $C^r_\K$ (being a composition of $C^r_\K$-maps).
Assume, conversely, that $f\circ q$ is $C^r_\K$.
For each $y\in N$, we find $x\in M$ such that $q(x)=y$
and (with Lemma~\ref{propsec}) a $C^r_\K$-map $\sigma\colon P\to M$
on an open neighbourhood $P$ of $q(x)=y$ in~$N$
such that $\sigma(y)=x$ and $q\circ \sigma=\id_P$.
Thus $f|_P=f|_P \circ \id_P=(f\circ q)\circ \sigma$ is $C^r_\K$.
As the sets $P$ just obtained cover $N$ for $y$ ranging through $N$,
we see that $f$ is $C^r_\K$.
\end{proof}
\begin{la}\label{uniquesub}
Let $M$ be a $C^r_\K$-manifold and $q\colon M\to N$ be a surjective map.
If $N$ can be made a $C^r_\K$-manifold $N_j$ in such a way that $q\colon M\to N_j$
is a $C^r_\K$-submersion for $j\in \{1,2\}$, then $\id_N\colon N_1\to N_2$
is a $C^\infty_\K$-diffeomorphism.
In particular, if both $N_1$ and $N_2$ are modelled on the same locally convex space~$F$,
then $N_1=N_2$.
\end{la}
\begin{proof}
Write $q_j$ for $q$, considered as a $C^r_\K$-submersion to $N_j$.
Since $\id_N\circ q_1=q_2$ is $C^r_\K$ and $q_1$ is a surjective
$C^r_\K$-submersion, Lemma~\ref{checkdiff} shows that
$\id\colon N_1\to N_2$ is $C^r_\K$.
Likewise, $\id\colon N_2\to N_1$ is $C^r_\K$.
Thus $\id\colon N_1\to N_2$ is a $C^r_\K$-diffeomorphism.
If $N_1$ and $N_2$ are modelled on the same space $F$,
this means that both have the same atlas, whence they
coincide as $C^r_\K$-manifolds.
\end{proof}
\begin{la}\label{subatx}
Let $f\colon M\to N$ be a $C^r_\K$-map between $C^r_\K$-manifolds
and $x_0\in M$ be a point such that the continuous linear map $T_{x_0}f\colon T_{x_0}M\to T_{f(x_0)}N$
has a continuous linear right inverse.
If $N$ is finite-dimensional or $N$ is a Banach manifold and $r\geq 2$,
then $x_0$ has an open neighbourhood $U\sub M$ such that $f|_U$ is a $C^r_\K$-submersion.
\end{la}
\begin{proof}
As the assertion is local, we
may assume that $M$ and $N$ are open subsets of their modelling space~$E$
and $F$, respectively. Let $\sigma\colon F\to E$ be a continuous linear right inverse
to $f'(x_0)$. Then $E=\ker f'(x_0)\oplus\sigma(F)$ as a topological
vector space and $\alpha:=f'(x_0)|_{\sigma(F)}\colon \sigma(F)\to F$
is an isomorphism of topological vector spaces.
After replacing $f$ with $\alpha^{-1}\circ f$, we may assume that $F=\sigma(F)$ and
that $f'(x_0)$ is the projection $\pr_2\colon \ker f'(x_0)\oplus F\to F$.
Thus $f'(x_0)|_F\colon F\to F$ is invertible.
After shrinking the open $X_0$-neighbourhood $M_1$, we may assume that $M_1=U\times V$
with open subsets $U\sub \ker(f'(x_0))$ and $V\sub F$.
By the Inverse Function Theorem with Parameters (see \ref{invpara}),
after shrinking $U$ and $V$ we may assume that $f(x,.)\colon V\to F$
is a $C^r_\K$-diffeomorphism onto an open subset of~$F$ for each $x\in U$.
As a consequence, $f(x,.)'(y)=f'(x,y)|_F$ is an isomorphism
of topological vector spaces for all $(x,y)\in U\times V=M_1$.
Holding $(x,y)$ fixed, abbreviate $\beta:=f'(x,y)|_F$.
Then the continuous linear map $\beta^{-1}\colon F\to F\sub E$
is a right inverse to $f'(x,y)$ as $f'(x,y)\circ \beta=f'(x,y)|_F\circ\beta=\id_F$.
Thus $f$ is a $C^r_\K$-submersion.
\end{proof}
\begin{la}\label{alwayszero}
Let $N$ be a $C^r_\K$-submanifold of finite co-dimension $k$ in a $C^r_\K$-manifold~$M$.
Then, for each $x\in N$, there exists an open neighbourhood $U$ of $x$ in $M$
and a $C^r_\K$-submersion $g\colon U\to\R^k$ such that
$N\cap U=g^{-1}(\{0\})$.
\end{la}
\begin{proof}
Let $E$ be the modelling space of~$M$ and $F\sub E$ be a closed vector subspace
such that $N$ is modelled on~$F$. There is a chart $\phi\colon U\to V\sub E$ of~$M$
such that $x\in U$ and $\phi(N\cap U)=F\cap V$.
Since $F$ has finite codimension $k$ in $E$, there exists a vector subspace $Y\sub E$
of dimension $k$ such that $E=F\oplus Y$ as a topological vector space.
Let $\pr_2\colon F\oplus Y\to Y$ be the projection onto the second component
and $\alpha\colon Y\to \R^k$ be an isomorphism.
Then $g:=\alpha\circ \pr_2\circ \phi\colon U\to \R^k$ is a $C^r_\K$-map
such that $N\cap U=g^{-1}(\{0\})$. Since $\alpha^{-1}\circ g\circ \phi^{-1}=\pr_2|_V$,
we see that $g$ is a $C^r_\K$-submersion.
\end{proof}
\begin{la}\label{locdeco}
Let $f\colon M\to N$ be a $C^r_\K$-submersion
and $x_0\in M$.
Then there exists an open neighbourhood
$O$ of $x_0$ in~$M$, an open subset $S\sub \ker T_{x_0}f$
and a $C^r_\K$-diffeomorphism
\[
\theta\colon f(O)\times S\to O
\]
such that $\theta(\{y\}\times S)=O\cap f^{-1}(\{y\})$ for all $y\in f(O)$.
\end{la}
\begin{proof}
There exists a chart $\phi\colon U_\phi\to V_\phi\sub E$ of~$M$ with $x:0\in U_\phi$
and a chart $\psi\colon U_\psi\to V_\psi\sub F$
such that $f(U_\phi)\sub U_\psi$ and $\psi\circ f\circ \phi^{-1}=\pi|_{V_\phi}$
for a continuous linear map $\pi\colon E\to F$ which admits a continuous linear right
inverse $\sigma\colon F\to E$. Then $E=\ker(\pi)\oplus \sigma(F)$ as a topological
vector space. After shrinking the neighbourhood $V_\phi$ of $\phi(x_0)$,
we may assume that $V_\phi=P\times Q$ with open subsets $P\sub\ker\pi$ and $Q\sub \sigma(F)$.
After shrinking $V_\phi$, we may assume that $V_\phi=\pi(Q)$.
Thus $\pi|_Q\colon Q\to V_\psi$ is a $C^r_\K$-diffeomorphism,
with inverse $\sigma|_{V_\psi}\colon V_\psi\to Q$.
The map
\[
\alpha:=d\phi|_{T_{x_0}M} \colon T_{x_0}M \to E
\]
is an isomorphism of topological vector spaces and $\alpha^{-1}(\ker\pi)=\ker T_{x_0}f$.
Define $O:=U_\phi$, $S:=\alpha^{-1}(Q)\sub \ker T_{x_0}f$ and
\[
\theta\colon f(O)\times S\to O,\quad \theta(y,z):=\phi^{-1}(\alpha(z)+\sigma(\psi(y))).
\]
Since $V_\phi=P\times Q$, the map $\theta$ is a $C^r_\K$-diffeomorphism.
For $(y,z)\in f(O)\times S$, we have
\begin{eqnarray*}
f(\theta(y,z))& =& \psi^{-1}\psi f\phi^{-1}(\alpha(z)+\sigma(\psi(y)))=\psi^{-1}\pi(\alpha(z)+\sigma(\psi(y)))\\
&=&\psi^{-1}(\pi\sigma(\psi(y)))=\psi^{-1}(\psi(y))=y.
\end{eqnarray*}
Hence $\theta(\{y\}\times S)=O\cap f^{-1}(\{y\})$.
\end{proof}
\begin{la}\label{easieremb}
For a $C^r_\K$-map $f\colon M\to N$ between $C^r_\K$-manifolds,
the following conditions are equivalent:
\begin{itemize}
\item[\rm(a)]
$f$ is a $C^r_\K$-embedding;
\item[\rm(b)]
$f(M)$ is a $($not necessarily pure$)$ split $C^r_\K$-submanifold of~$N$ and $f|^{f(M)}\colon M\to f(M)$ is
a $C^r_\K$-diffeomorphism.
\end{itemize}
\end{la}
\begin{proof}
Let $E_1$ be the modelling space of $M$ and $E_2$ be the modelling space
of~$N$.
Let $j\colon f(M)\to N$ be the inclusion map.

(b)$\impl$(a).
Assume that $f(M)$ is a split $C^r_\K$ submanifold of~$N$, modelled on a complemented vector subspace
$F\sub E_2$,
and assume that $f|^{f(M)}\colon M\to f(M)$ is a $C^r_\K$-diffeomorphism.
Then $f=j\circ f|^{f(M)}$ is a topological embedding.
Moreover, $E_1$ is isomorphic to~$F$,
and we choose an isomorphism
\[
\alpha\colon E_1\to F
\]
of topological vector spaces.
Let $\psi\colon U_\psi\to V_\psi\sub E_2$ be a chart for $N$
such that $f(x)\in U_\psi$ and
\[
\psi(f(M)\cap U_\psi)= F\cap V_\psi.
\]
Let $Y\sub E_2$ be a vector subspace such that $E_2=F\oplus Y$ as a topological vector space.
After shrinking the open neighbourhood $U_\psi$ of $f(x)$,
we may assume that $V_\psi=P\times Q$ with open subsets $P\sub F$ and $Q\sub Y$.
Thus $F\cap V_\psi=P$.
Let $\phi\colon U_\phi\to V_\phi\sub E_1$ be a chart
of~$M$ such that $x\in U_\phi$ and $f(U_\phi)\sub U_\psi$.
Since $\psi|_{f(M)\cap U_\psi}\colon f(M)\cap U_\psi\to F\cap V_\psi$ is a chart for $f(M)$, 
\[
\psi\circ f\circ \phi^{-1}\colon V_\phi\to F\cap V_\psi
\]
is a $C^r_\K$-diffeomorphism onto an open subset of $F\cap V_\psi=P$.
After shrinking~$P$, we may assume that $\psi\circ f\circ \phi^{-1}$
is a $C^r_\K$-diffeomorphism $V_\phi\to P$.
Let $W:=\alpha^{-1}(P)\sub E_1$.
Then
\[
\theta:=(\psi\circ f\circ \phi^{-1})^{-1}\circ \alpha|_W\colon W\to V_\phi
\]
is a $C^r_\K$-diffeomorphism and hence also $\theta^{-1}\circ \phi\colon U_\phi\to W$ is a chart
for~$M$. Now
\[
\psi\circ f\circ (\theta^{-1}\circ \phi)^{-1}=\psi\circ f\circ \phi^{-1}\circ \theta=\alpha|_W
 \]
where $\alpha\colon E_1\to F\sub E_2$ is a linear topological embedding
onto a complemented vector subspace of~$E_2$.
Hence $f$ is a $C^r_\K$-immersion.

(a)$\impl$(b). If $f$ is a $C^r_\K$-embedding,
then $f$ is a topological embedding and a $C^r_\K$-immersion.
Let $y\in f(M)$ and $x\in M$ such that $y=f(x)$.
We find
charts $\phi\colon U_\phi\to V_\phi \sub E_1$
and $\psi\colon U_\psi\to V_\psi\sub E_2$ of $M$ and $N$, respectively,
such that $x\in U_\phi$, $f(U_\phi)\sub U_\psi$ and
\[
\psi\circ f\circ \phi^{-1}=i|_{V_\phi}
\]
for a linear topological embedding $i\colon E_1\to E_2$
onto a complemented vector subspace.
Set $F:=i(E_1)$ and choose a vector subspace
$Y\sub E_2$ such that $E_2=F\oplus Y$.
Since $i(V_\phi)$ is relatively open in $F\sub E_2$,
there is an open subset $Q\sub E_2$ such that $i(V_\phi)=F\cap Q$.
We may assume that $Q\sub V_\psi$;
after replacing $V_\psi$ with $Q$, we may assume that $V_\psi=Q$ and hence
\[
F\cap V_\psi=i(V_\phi).
\]
Since $f$ is a topological embedding, $f(U_\phi)$ is relatively open in $f(M)$, whence there exists an open subset
$W\sub N$ such that $W\cap f(M)=f(U_\psi)$.
We may assume that $W\sub U_\psi$;
after shrinking $U_\psi$,
we may assume that $U_\psi=W$ and hence
\[
f(M) \cap U_\psi = f(U_\phi)=\psi^{-1}((\psi\circ f\circ \phi^{-1})(V_\phi))=\psi^{-1}(i(V_\phi))
=\psi^{-1}(F\cap V_\psi).
\]
Thus $\psi(f(M)\cap U_\psi)=F\cap V_\psi$,
showing that $f(M)$ is a (not necessarily pure) $C^r_\K$-submanifold of~$N$
with $f(M)\cap U_\psi$ a $C^r_\K$-submanifold modelled on~$F$.
As $F$ is complemented in~$E_2$, we are dealing with split submanifolds.
With respect to the chart $\phi$ for $M$ and the chart
\[
\tau:=\psi|_{f(M)\cap U_\psi}\colon
f(M)\cap U_\psi\to F\cap V_\psi
\]
for $f(M)$, we have $\tau\circ f\circ \phi^{-1}=i|_{V_\phi}\colon V_\phi\to i(V_\phi)$.
As this map is a $C^r_\K$-diffeomorphism, $f\colon M\to f(M)$ is a local $C^r_\K$-diffeomorphism
around each point and hence a $C^r_\K$-diffeomorphism,
as $f\colon M\to f(M)$ also is a homeomorphism.
\end{proof}
\begin{la}\label{loclemb}
If $f\colon M\to N$ is a $C^r_\K$-immersion between $C^r_\K$-manifolds,
then each $x\in M$ has an open neighbourhood $W\sub M$ such that $f|_W$
is a $C^r_\K$-embedding.
\end{la}
\begin{proof}
Let $E_1$ and $E_2$ be the modelling space of $M$ and $N$, respectively.
Given $x\in M$, we find charts $\phi\colon U_\phi\to V_\phi\sub E_1$ and
$\psi\colon U_\psi\to V_\psi\sub E_2$ of $M$ and $N$, respectively,
such that $x\in U_\phi$, $f(U_\phi)\sub U_\psi$ and
\[
\psi\circ f\circ \phi^{-1}=j|_{V_\phi}
\]
for a linear topological embedding $j\colon E_1\to E_2$ with complemented
image. Hence $f|_{U_\phi}=\psi^{-1}\circ j|_{V_\phi}\circ \phi$
is a topological embedding and thus $f|_{U_\phi}$
is a $C^r_\K$-embedding.
\end{proof}
\begin{la}
Let $f\colon M_1\to M_2$ and $g\colon M_2\to M_3$ be $C^r_\K$-maps
between $C^r_\K$-manifolds.
\begin{itemize}
\item[\rm(a)]
If $f$ and $g$ are $C^r_\K$-immersions, then also $g\circ f\colon
M_1\to M_3$ is a $C^r_\K$-immersion.
\item[\rm(b)]
If $f$ and $g$ are $C^r_\K$-embeddings, then also $g\circ f\colon M_1\to M_3$
is a $C^r_\K$-embedding.
\end{itemize}
\end{la}
\begin{proof}
(b) By Lemma~\ref{easieremb},
$f(M_1)$ is a split $C^r_\K$-submanifold of~$M_2$,
$g(M_2)$ is a split $C^r_\K$-submanifold of~$M_3$ and
both $f\colon M_1\to f(M_1)$ and $g\colon M_2\to g(M_2)$ is a $C^r_\K$-diffeomorphism.
Thus $g(f(M))$ is a split $C^r_\K$-submanifold of $g(M_2)$.
By Lemma~\ref{subinsub}, $g(f(M))$ is a $C^r_\K$-submanifold of~$M_3$,
and both $g(M_2)$ and $M_3$ induce the same $C^r_\K$-manifold structure
on $g(f(M_1))$. Since $g|_{f(M_1)}\colon f(M_1)\to g(f(M_1))\sub g(M_2)$ is a $C^r_\K$-diffeomorphism,
by the preceding it also is a $C^r_\K$-diffeomorphism to $g(f(M_1))\sub M_3$, and hence
$g\circ f|_{M_1}=g|_{f(M_1)}\circ f_{M_1}\colon M_1\to g(f(M_1))\sub M_3$
is a $C^r_\K$-diffeomorphism.
Thus $g\circ f$ is a $C^r_\K$-embedding, by Lemma~\ref{easieremb}.

(a) Let $x\in M_1$. By Lemma~\ref{loclemb},
there are open neighbourhoods $W_1\sub M_1$ of $x$ and $W_2\sub M_2$ of $f(x)$
such that $f|_{W_1}$ and $g|_{W_2}$ are $C^r_\K$-embeddings.
After shrinking $W_1$, we may assume that $f(W_1)\sub W_2$.
Then $(g\circ f)|_{W_1}=g|_{W_2}\circ f|_{W_1}$ is a $C^r_\K$-embedding (by (b) already established)
and hence $(g\circ f)|_{W_1}$ is a $C^r_\K$-immersion.
As the open sets $W_1$ cover $M_1$ as $x$ ranges through $M_1$, we see that $g\circ f$
is a $C^r_\K$-immersion.
\end{proof}
\begin{la}
Let $f\colon M_1\to N_1$ and $g\colon M_2\to N_2$ be $C^r_\K$-maps
between $C^r_\K$-manifolds.
\begin{itemize}
\item[\rm(a)]
If $f$ and $g$ are $C^r_\K$-immersions, then also $g\times f\colon
M_1\times M_2 \to N_1\times N_2$ is a $C^r_\K$-immersion.
\item[\rm(b)]
If $f$ and $g$ are $C^r_\K$-embeddings, then also $g\times f\colon M_1\times M_2\to N_1\times N_2$
is a $C^r_\K$-embedding.
\end{itemize}
\end{la}
\begin{proof}
Let $E_j$ be the modelling space of~$M_j$ and $F_j$ be the modelling space of $N_j$ for $j\in \{1,2\}$.

(a) $f$ locally looks like a linear topological embedding $i_1\colon E_1\to F_1$ with complemented image
and $g$ locally looks like a linear topological embedding $i_2\colon E_2\to F_2$ with complemented
image. Hence $f\times g$ locally looks like $i_1\times i_2$ which is a linear topological
embedding with complemented image. Thus $f\times g$ is a $C^r_\K$-immersion.

(b) $f\times g$ is a $C^r_\K$-embedding (by (a)) and a topological embedding,
as both $f$ and $g$ is a topological embedding. Hence $f\times g$ is a $C^r_\K$-embedding.
\end{proof}
\section{Proof of Theorem A}
We know that every $C^r_\K$-submersion is a na\"{\i}ve submersion. Conversely,
assume that the $C^r_\K$-map $f\colon M\to N$ is a na\"{\i}ve
submersion to the Banach manifold~$N$.
Since being a $C^r_\K$-submersion is
a local property, we may assume that $W:=M$ is an open subset of a locally
convex space $E$ and $N$ is an open subset of a Banach space~$Y$. 
Let $w_0\in W$. Since $f$ is a na\"{\i}ve submersion,
there exists a vector subspace $F\sub E$ such that $E=(\ker f'(w_0))\oplus F$
as a topological vector space and $\alpha:=f'(w_0)|_F\colon F\to Y$ is an
isomorphism of topological vector spaces.
After replacing $f$ with $\alpha^{-1}\circ f$, we may assume that $Y=F$
(and $f'(w_0)|_F=\id_F$). Write $w_0=(x_0,y_0)$ with $x_0\in \ker f'(w_0)$ and $y_0\in F$.
After shrinking $W$, we may assume that $W=U\times V$
with open neighbourhoods $U\sub \ker f'(w)$ of $x_0$ and $V\sub F$ of $y_0$.
By the inverse function theorem with parameters (see \ref{invpara}),
after shrinking the neighbourhoods $U$ and $V$ we may assume that
the map
\[
\theta\colon U\times V\to U\times F,\quad (x,y)\mto(x,f(x,y))
\]
is a $C^r_\K$-diffeomorphism onto an open subset $\Omega$ of $U\times F$.
Identifying $E$ with $(\ker f'(w_0))\times F$, we can consider $\theta$ as a global chart
for $W=U\times V$. On $N\sub F$, we use the the chart $\id_N$.
Then
\[
\id_N\circ f\circ \theta^{-1}=f\circ \theta^{-1}=\pr_2\circ \theta\circ\theta^{-1}=\pr_2|_\Omega
\]
with $\pr_2\colon E\times F\to F$, $(x,y)\mto y$. Thus $f$ is a $C^r_\K$-submersion.\,\Punkt
\section{Proof of Theorem B}
Let us assume that $g\colon M_2\to N$ is a $C^r_\K$-submersion in the situation
of Theorem~B (the case that $f$ is a $C^r_\K$-submersion is similar).
Then, for each $(x_0,y_0)\in S:=\{(x,y)\in M_1\times M_2\colon f(x)=g(x)\}$,
we find a chart $\phi\colon U_\phi\to V_\phi$ from an open neighbourhood $U_\phi\sub M_2$ of $y_0$
onto an open set $V_\phi$ in the modelling space $E_2$ of $M_2$
and a chart $\psi\colon U_\psi\to V_\psi$ from an open neighbourhood $U_\psi\sub N$ of $g(y_0)$
onto an open subset $U_\psi$ of the modelling space $F$ of $N$ such that $g(U_\phi)\sub U_\psi$
and
\[
\psi\circ g\circ \phi^{-1}=\pi|_{V_\phi}
\]
for some continuous linear map $\pi\colon E_2\to F$
which has a continuous linear right inverse $\sigma\colon F\to E_2$.
After replacing the modelling space of $N$ with the isomorphic locally convex space
$\sigma(F)$, we may assume that $E_2=\ker(\pi)\oplus F$ as a topological vector space
and $\pi=\pr_1\colon F\times \ker(\pi)\to F$ is the projection onto the first factor.
There is a chart $\tau\colon U_\tau\to V_\tau$ of $M_1$ with $x_0\in U_\tau$ and $f(U_\tau)\sub U_\psi$.
Since being a submanifold is a local property, it suffices to show that $T:=S\cap (U_\tau\times U_\phi)$
is a submanifold of $U_\tau\times U_\phi$.
Equivalently, we need to show that $(\tau\times \phi)(T)$ is a submanifold
of $V_\tau\times V_\phi$. But this is the fibre product of $V_\tau$ and $V_\phi$
with respect to the maps $\psi\circ f\circ \tau^{-1}$ and $\psi\circ g\circ \phi^{-1}$.
We may therefore assume now that $M_j$ is an open subset of a locally convex space~$E_j$
for $j\in \{1,2\}$, that $N$ an open subset of a complemented vector subspace $F\sub E_2$,
and $g=\pr_1|_{M_2}$ with $\pr_1\colon F\oplus C\to F$, $(y,z)\mto y$
for a complementary vector subspace $C\sub E_2$. After shrinking the
neighbourhood $M_2$ of the given point $y_0$, we may assume
that $M_2=P\times Q$ with
For $x\in M_1$ and $(y,z)\in M_2$ with $y\in F$ and $z\in C$, we have
\[
f(x)=g(y,z)=\pr_1(y,z)=y
\]
if and only if $y=f(x)$. Thus $S=\graph(f)\times Q$ is a split $C^r_\K$-submanifold
of $M_1\times M_2=M_1\times P\times Q$ (recalling that graphs are split). Let
$\pi\colon M_1\times M_2\to M_1$ denote the pojection
onto the first component. The map $\theta\colon S\to M_1\times Q$, $(x,y,z)\to (x,z)$
is a chart for $S$ with inverse
\[
\theta^{-1}\colon M_1\times Q\to S,\quad (x,z)\mto (x,f(x),z).
\]
Using the chart $\id_{M_1}$ for $M_1$, we obtain $(\id_{M_1}\circ \pi\circ \theta^{-1})(x,z)=x$.
Thus $\pi|_S\colon S\to M_1$ is a $C^r_\K$-submersion.\,\Punkt
\section{Proof of Theorem C}
Let $E$ be the modelling space of $M$ and $F$ be the modelling space of~$N$.
Given $x_0\in M$, let $\phi\colon U_\phi\to V_\phi\sub E$
be a chart of~$M$ with $x_0\in U_\phi$ and $\psi\colon U_\psi\to V_\phi\sub F$
be a chart for $N$ with $f(x_0)\in U_\psi$ such that $f(U_\phi)\sub U_\psi$
and
\[
\psi\circ f\circ \phi^{-1}=\pi|_{V_\phi}
\]
where $\pi\colon E\to F$ is a continuous linear map which admits a continuous linear
right inverse $\sigma\colon F\to E$. It suffices to show that
$f^{-1}(S)\cap U_\phi=(f|_{U_\phi})^{-1}(S\cap U_\psi)$
is a submanifold of $U_\phi$ or, equivalently, that
$\phi(f^{-1}(S)\cap U_\phi)=(\psi\circ f\circ \phi^{-1})^{-1}(\psi(S\cap U_\psi))$
is a submanifold of $V_\phi$.
We may therefore assume that $M$ and $N$ are open subsets of $E$
and $F$, respectively, and that $f=\pi|_M$
for a continuous linear map $\pi\colon E\to F$ with a continuous linear right inverse $\sigma\colon F\to E$.
After replacing $F$ with $\sigma(F)$, we may assume that $\pi\colon F\oplus \ker(\pi)\to F$
is the projection onto the first factor.
After shrinking the open $x_0$-neighbourhood $M_1$, we may assume that $M_1=P\times Q$ with open
subsets $P\sub F$ and $Q\sub \ker(\pi)$. Then $P\sub M_2$;
after shrinking $M_2$, we may assume that $M_2=P$.
Let $Y\sub F$ be the modelling space of~$S$.
Now
\[
f^{-1}(S)=S\times Q
\]
is a $C^r_\K$-submanifold of $P\times Q=M_2\times Q=M_1$
(which is split if $S$ is split)
and $g:=f|_{f^{-1}(S)}\colon S\times Q\to S$ is a $C^r_\K$-submersion because
for each chart $\theta\colon U_\theta\to V_\theta\sub Y$ of $S$, we have
$\theta\circ g\circ (\theta\times \id_Q)^{-1}=\pr_1|_{V_\theta\times Q}$
with the projection $\pr_1\colon Y\times \ker(\pi)\to Y$.
If $S$ has finite codimension $k$ in $N$, then $\dim_\K(F/Y)=k$.
Hence also $\dim_\K((F\times \ker(\pi))/(Y\times \ker(\pi))=k$
and thus $f^{-1}(S)=S\times Q$ has finite codimension, $k$,
in $P\times Q=M_1$.
Finally, if $x=(s,q)\in f^{-1}(S)$, then
$T_x(f^{-1}(S))=T_sS\times T_qQ=T_sS\times \ker\pi=\pi^{-1}(T_sS)$
(identifying $T_sS$ with a vector subspace of~$F$).\,\Punkt
\section{Proof of Theorem D}
By the hypothesis an Lemma~\ref{subatx}, each $x\in f^{-1}(\{y\})$
has an open neighbourhood $U_x\sub M$ such that $f|_{U_x}$ is a $C^r_\K$-submersion.
Then
\[
U:=\bigcup_{x\in f^{-1}(\{y\})}U_x
\]
is an open submanifold of $M$ such that $f^{-1}(\{y\})\sub U$
and $g:=f|_U$ is a $C^r_\K$-submersion.
Since $\{y\}$ is a split submanifold of $N$ with $T_y\{y\}=\{0\}$, we deduce with Theorem~$C$
that $f^{-1}(\{y\})=g^{-1}(\{y\})$ is a (not necessarily pure)
split $C^r_\K$-submanifold of~$U$ and hence of~$M$,
with $T_x f^{-1}(\{y\})=T_xg^{-1}(\{0\})=\ker T_xg=\ker T_xf$
for each $x\in f^{-1}(\{y\})$.
If $N$ has finite dimension $k$, then $\{y\}$ has co-dimension $k$ in~$N$
and thus $g^{-1}(\{y\})$ has co-dimension $k$ in $U$
(and hence in $M$), by Theorem~C.\,\Punkt
\section{Proof of Theorem E}
For each $x\in f^{-1}(S)$, $T_xf\colon T_xM\to T_{f(x)}N$ has dense image.
Since $S$ is a $C^r_\K$-submanifold of~$N$ of finite codimension~$k$,
there exists an open neighbourhood $W\sub N$ of $f(x_0)$
and a $C^r_\K$-submersion $g\colon W\to \R^k$ such that
$S\cap W=g^{-1}(\{0\})$ (see Lemma~\ref{alwayszero}).
Then $f^{-1}(W)$ is an open neighbourhood of $x_0$ in~$M$
such that
\[
f^{-1}(S)\cap f^{-1}(W)=f^{-1}(S\cap W)=f^{-1}(g^{-1}(\{0\}))
=(g\circ f|_{f^{-1}(W)})^{-1}(\{0\}).
\]
Now $h:=g\circ f|_{f^{-1}(W)}\colon f^{-1}(W)\to\R^k$ is a $C^r_\K$-map
such that $T_yh=T_{f(y)}g\circ T_zf$ has dense image
for each $y\in h^{-1}(\{0\})=f^{-1}(S)\cap f^{-1}(W)$.
As every dense vector subspace of $T_{h(y)}(\R^k)\cong\R^k$
is all of $T_{h(y)}(\R^k)$,
we deduce that $T_yh\colon T_y M\to T_{h(y)}(\R^k)$ is surjective
for each $y\in f^{-1}(S)\cap f^{-1}(W)$.
Hence, by the regular value theorem with finite-dimensional
range manifolds (as stated after Theorem~D),
$f^{-1}(S)\cap f^{-1}(W)=h^{-1}(\{0\})$
is a split $C^r_\K$-submanifold of $f^{-1}(W)$ of codimension~$k$,
and $T_x(f^{-1}(S))=\ker T_xh=(T_xf)^{-1}(\ker T_{f(x)} g)=(T_xf)^{-1}(T_{f(x)}S)$
for each $x\in f^{-1}(S)\cap f^{-1}(W)$.
Since $f^{-1}(W)$ is an open neighbourhood of $x_0$ in~$M$,
the assertion follows.\,\Punkt
\section{Proof of Theorem F}
See, e.g., \cite[Part II, Chapter III, \S10 4) Theorem]{Ser}
for a proof of the
constant rank theorem
for mappings between open subsets of finite-dimensional
vector spaces. The following proof is also indebted to~\cite{Gai},
where a particularly clear presentation of the classical case is given.\\[2mm]
Because all assertions of Theorem~F are of a local nature,
it suffices to prove it in the case that $M=U$ is an open subset
of a locally convex space~$E$ and $N=\R^n$.
Let $x_0\in U$ and consider a $C^r_\K$-map $f\colon U\to\R^n$ such that $\im f'(x)$ has dimension~$\ell$,
for each $x\in U$. Let $Z:=\im f'(x_0)$. After replacing $f$ with $\alpha\circ f$ with a suitable
automorphism $\alpha\colon \R^n\to\R^n$, we may assume that $Z=\R^\ell\times\{0\}$,
which we identify with $\R^\ell$.
Then $K:=\ker f'(x_0)$ has finite co-dimension, enabling us to pick a finite-dimensional vector subspace
$C\sub E$ such that $E=C\oplus K$. Then $Z=f'(x_0)(C)$,
whence we can pick an $\ell$-dimensional vector subspace
$Y\sub C$ with $f'(x_0)(Y)=Z$.
Let $D\sub C$ be a vector subspace such that $C=Y\oplus D$.
Then
\[
E=Y\oplus D\oplus K =Y\oplus W
\]
with $W:=D\oplus K$
and
\[
f'(x_0)|_Y \colon Y\to Z=\R^\ell\times\{0\}
\]
is an isomorphism onto its image $Z$.
Write $f=(f_1,f_2)$ with $f_1\colon U\to\R^\ell$, $f_2\colon U\to \R^{n-\ell}$.
Then $(f_1)'(x_0)|_Y=f'(x_0)|_Y^Z$ is an isomorphism between $\ell$-dimensional
vector spaces,
showing that the $C^r_\K$-map $f_1$
satisfies the hypotheses of the Inverse Function Theorem with Parameters
(cf.\ \ref{invpara}). We therefore find an open neighbourhood
$P\times Q\sub Y\times W$ of $x_0$ such that the map
\[
g\colon P\times Q\to Z\times W,\quad (y,w)\mto (f_1(y,w),w)
\]
has open image $\Omega$ and is a $C^r_\K$-diffeomorphism onto
its image. Write $x_0=(y_0,w_0)$ with $y_0\in P\sub Y$, $w_0\in Q\sub W$.
We choose a connected open neighbourhood $\Omega_2$ of $w_0$ in the locally convex space~$W$
and an open neighbourhood
$\Omega_1$ of $f_1(y_0,w_0)$ in~$\R^\ell$ such that
\[
\Omega_1\times \Omega_2 \sub \Omega.
\]
Then $\Omega_3:=g^{-1}(\Omega_1\times \Omega_2)$ is an open neighbourhood
of $x_0$ in $U$ and
$g$ restricts to the $C^r_\K$-diffeomorphism
\[
\phi \colon \Omega_3 \to \Omega_1\times \Omega_2,\quad \phi(y,w)=(f_1(y,w),w).
\]
Then
\[
(f\circ \phi^{-1})(v,w)=(v,h(v,w))\quad \mbox{for all $(v,w)\in\Omega_1\times \Omega_2$}
\]
with a $C^r_\K$-function $h\colon \Omega_1\times\Omega_2\to \R^{n-\ell}$.
By the preceding formula,
for all $(v,w)\in\Omega_1\times\Omega$,
we can express the image of the partial differential with respect to the variables in $\R^\ell$ as
\[
\im\, d_1(f\circ \phi^{-1})(v,w;.)=\graph d_1h(v,w,.),
\]
which is an $\ell$-imensional vector space and intersects
$\{0\}\times \R^{n-\ell}$ in the trivial vector subspace.
Hence, because $\im f'(v,w)$ has dimension $\ell$,
it must not contain non-zero vectors $u$ in $\{0\}\times \R^{n-\ell}$,
as the sum $\graph d_1h(v,w,.)+\R v$ would be direct and hence
$\dim \,\im\,f'(v,w)>\ell$, contradicting the hypotheses.
We deduce that
\[
d_2h(v,w,.)=0\;\;\mbox{for all $(v,w)\in\Omega_1\times\Omega_2$.}
\]
Thus, for fixed $x$, the map $h(x,.)\colon\Omega_2\to\R^{n-k}$
has vanishing differential and so $h(x,.)$ is constant, as $\Omega_2$
is assumed connected (see, e.g., \cite{GaN}).
Hence $h(v,w)=h(v,w_0)=b(v)$
in terms of the $C^r_\K$-function
\[
b\colon \Omega_1\to \R^{n-\ell},\quad b(v):=h(v,w_0).
\]
Thus $(f\circ\phi^{-1})(v,w)=(v,b(v))$ for all $v\in \Omega_1$.
Now
\[
\psi\colon \R^\ell\times\R^{n-\ell}\to\R^\ell\times\R^{n-\ell},
\;\; (v,u)\mto (v,u-b(v))
\]
is a $C^r_\K$-diffeomorphism (with inverse $(v,u)\mto (v,u+b(v))$)
and
\[
(\psi\circ f\circ\phi^{-1})(v,w)=(v,0),
\]
as desired.\\[2.3mm]
Note that $f(\Omega_3)$ is a $C^r_\K$-submanifold of
$\R^n$ as $\psi$ is a diffeomorphism and $\psi(f(\Omega_3))=(\psi \circ f\circ \phi^{-1})(\Omega_1\times\Omega_2)=
\Omega_1\times \{0\}$ is a $C^r_\K$-submanifold of $\R^\ell\times \R^{n-\ell}=\R^n$.
Moreover, $\psi$ restricts to a $C^r_\K$-diffeomorphism
$f(\Omega_3)\to \Omega_1\times\{0\}=g(\Omega_1\times\Omega_2)$
with
\[
g:=\psi\circ f\circ \phi^{-1}\colon \Omega_1\times\Omega_2\to\R^n=\R^\ell\times\R^{n-\ell},\;\;
(v,w)\mto (v,0).
\]
For each $y\in f(\Omega_3)$, we have
\[
\Omega_3\cap f^{-1}(\{y\})=  \phi^{-1}( (\psi\circ f\circ \phi^{-1})^{-1}( \{\psi(y)\})
=\phi^{-1}(g^{-1}(\{\psi(y)\})).
\]
Hence $\Omega_3 \cap f^{-1}(\{y_0\})$ will be a split $C^r_\K$-submanifold of $\Omega_3$
with tangent spaces as asserted if $g^{-1}(\{\psi(y)\})$ is a split $C^r_\K$-submanifold
of $\Omega_1\times\Omega_2$ and $T_x (g^{-1}(\{\psi(y\}))=\ker T_xg$
for each $x\in g^{-1}(\{\psi(y)\})$.
We have $y=\phi^{-1}(v,w)$ with suitable $v\in\Omega_1$ and $w\in \Omega_2$.
Then $\psi(y)=g(v,w)=(v,0)$ and
\[
g^{-1}(\{\psi(y)\})=g^{-1}(v,0)=\{v\}\times\Omega_2,
\]
which indeed is a split $C^r_\K$-submanifold of $\Omega_1\times \Omega_2$
with tangent space $\{0\}\times W=\ker T_xg$ for each~$x$.
The map
\[
\zeta\colon \im(g) \times\Omega_2=
(\Omega_1\times \{0\})\times \Omega_2\to \Omega_1\times \Omega_2,\;
((v,0),w)\mto (v,w)
\]
is a $C^r_\K$-diffeomorphism such that
\[
\zeta(\{(v,0)\}\times\Omega_2)=\{v\}\times\Omega_2=g^{-1}(\{(v,0)\})
\]
for each $(v,0)\in \im(g)$.
Then
\[
\theta\colon f(\Omega_3)\times \Omega_2\to\Omega_3,\quad
\theta(y,w):=\phi^{-1}(\zeta(\psi(y),w))
\]
is a $C^r_\K$-diffeomorphism
such that $\theta(\{y\}\times \Omega_3)=\Omega_3\cap f^{-1}(\{y\})$ for each $y\in f(\Omega_3)$.\,\Punkt
\section{Proof of Theorem G}
In this section, our conventions are as follows:
We let $\K\in\{\R,\C\}$.
If $\K=\R$, we let $r\in \{\infty,\omega\}$.
If $\K=\C$, we let $r:=\omega$.\\[2.3mm]
We shall use a folklore fact,
part of which can also be found in \cite{BaN}.
As the result is known in special cases and the proof
does not contain surprises, we relegate it to an appendix
(Appendix~\ref{apphomog}).
\begin{la}\label{lafolk}
Let $G$ be a $C^r_\K$-Lie group
and $H\sub G$ be a $C^r_\K$-Lie subgroup.\footnote{Thus $H$ is locally closed in $G$
(like any submanifold) and hence closed, being a subgroup.}
Let $q\colon G\to G/H$, $x\mto xH$ be the canonical
map. Then the following conditions are equivalent:
\begin{itemize}
\item[\rm(a)]
There exists a $C^r_\K$-manifold structure on $G/H$ which turns $q$ into a $C^r_\K$-submersion.
\item[\rm(b)]
There is a submanifold $S\sub G$ with $e\in S$
such that $SH\sub G$ is open and the product map
\[
S\times H\to SH,\quad (x,y)\mto xy
\]
is a $C^r_\K$-diffeomorphism.
\item[\rm(c)]
There is a submanifold $S\sub G$ with $e\in S$
and an open identity neighbourhood $W\sub H$
such that $SW\sub G$ is open and the product map
\[
S\times W\to SW,\quad (x,y)\mto xy
\]
is a $C^r_\K$-diffeomorphism.
\end{itemize}
If {\em(a)--(c)} hold, then the action $G\times G/H\to G/H$, $(g,xH)\mto gxH$ is $C^r_\K$.
Also, $T_eq=L(H)$.
If, moreover, $N$ is a normal subgroup of~$G$, then $G/H$ is a $C^r_\K$-Lie group.
\end{la}
\begin{numba}
Recall that an \emph{immersed $C^r_\K$-submanifold}
of a $C^r_\K$-manifold~$M$ is a $C^r_\K$-manifold~$N$
such that $N\sub M$ as a set and each $x\in N$ has an open neighbourhood
$W\sub N$ such that $W$ is a split $C^r_\K$-submanifold of $M$
and $\id_W\colon W\to W$
is a $C^r_\K$-diffeomorphism if we consider $W$ on the left as an open submanifold of~$N$
and give $W$ on the right the $C^r_\K$-manifold structure induced by~$M$.
The the inclusion map $j\colon N\to M$ is $C^r_\K$ in particular
and $j_x\colon T_xN\to T_xM$ is a linear topological embedding onto
a complemented vector subspace of $T_xM$.
We identify $T_xN$ with its image in $T_xM$.
\end{numba}
\begin{numba}
If $M$ is a $C^r_\K$-manifold modelled on~$E$, we consider its
tangent bundle $TM$ as a $C^r_\K$-vector bundle over $M$ with typical fibre~$E$.
A (regular) \emph{vector distribution on~$M$}
is a $C^r_\K$-vector subbundle $D\sub TM$.
We say that the vector distribution is finite-dimensional, Banach and co-Banach-respectively,
if its typical fibre $F\sub E$ is finite-dimensional, a Banach space, resp.,
a co-Banach vector subspace $F\sub E$.
We abbreviate $D_p:=D\cap T_pM$ for $p\in M$. Thus $D_p\cong F$.
\end{numba}
We define (cf.\ \cite{Ey2}):
\begin{defn}
If $D$ is a vector distribution on a $C^r_\K$-manifold~$M$,
then each connected immersed $C^r_\K$-submanifold $N\sub M$
such that
\[
T_p N=D_p\quad\mbox{for all $p\in N$}
\]
is called an \emph{integral submanifold}
for~$D$.
An integral submanifold $N\sub M$ is called \emph{maximal}
if it has the following property:\\[2.3mm]
For every integral submanifold $N_1\sub M$ such that $N_1\cap N\not=\emptyset$,
we have $N_1\sub N$ and the inclusion map
$N_1\to N$ is $C^r_\K$.
Maximal integral submanifolds are also called \emph{leafs}.\\[2.3mm]
A chart $\phi\colon U\to V$ of~$M$ is called a \emph{Frobenius chart}
if we can write $V=P\times Q$ with an open set $P\sub C$ in a
vector subspace $C\sub E$ such that $E=C\oplus F$ as a topological vector space
and a connected open subset $Q\sub F$,
and the $C^r_\K$-submanifolds
\[
S_p:=\phi^{-1}(\{p\}\times Q)
\]
are integral manifolds for $D$, for all $p\in P$.
\end{defn}
In \cite{Ey1} and \cite{Ey2}, J. M. Eyni proved
Frobenius theorems for finite-dimensional, Banach
and co-Banach vector distributions
on $C^k$-manifolds (under suitable hypotheses),
and applied these to the integration problem for Lie subalgebras
(cf.\ also \cite{Hi1} and \cite{Tei}).
His methods show:
\begin{prop}\label{maineyni}
Let $G$ be a $C^r_\K$-Lie group
and $\ch\sub L(G)$ be a Lie subalgebra of the $\K$-Lie algebra $L(G)$.
Assume that
\begin{itemize}
\item[\rm(a)]
$\ch$ is co-Banach in~$L(G)$; or:
\item[\rm(b)]
$\ch$ is a Banach space, complemented in $L(G)$, and
there is a $C^r_\K$-map $\ve \colon \ch\to G$
such that $\R\to G$, $t\mto \ve(tv)$ is a group
homomorphism with $\frac{d}{dt}\Big|_{t=0}\ve(tv)=v$
for all $v\in\ch$.
\end{itemize}
Then the left invariant vector distribution on $G$ given by $D_g:=T\lambda_g(\ch)$
admits leafs $L_g$ with $g\in L_g$, for each $g\in G$.
The leaf $H:=L_e$ is an immersed $C^r_\K$-Lie subgroup of~$G$
with $L(H)=\ch$ and $L_g=gH$ for each $g\in G$.
Moreover, for each $g\in G$ there exists a
Frobenius chart around~$g$.
\end{prop}
\begin{proof}
The case $\K=\R$, $r=\infty$ is covered by \cite{Ey1} and \cite{Ey2}.
However, because the inverse function theorem with parameters works
just as well for real and complex analytic maps,
inspecting the proofs one finds that the proposition remains valid
in the real and complex analytic cases.
\end{proof}
Note that condition (b) of the Proposition~\ref{maineyni}
is satisfied with $\ve:=\exp_G|_\ch$ if $G$ has an exponential
function $\exp_G\colon \cg\to G$ which is $C^r_\K$.
\begin{prop}\label{noregu}
If the leaf $H:=L_e$ in Proposition~{\em\ref{maineyni}}
is a $C^r_\K$-submanifold of~$G$ and immersed $C^r_\K$-submanifold structure
coincides with the $C^r_\K$-manifold structure induced by~$G$,
then $G/H$ can be made a $C^r_\K$-manifold in such a way that
$q\colon G\to G/H$, $x\mto xH$ becomes a $C^r_\K$-submersion.
\end{prop}
\begin{proof}
Let $\phi\colon U\to V=P\times Q$ with $P\sub C$, $Q\sub F$ be a chart of $M$ which is a Frobenius chart
for the left invariant vector distribution given by $D_g:=T\lambda_g \ch$
(as provided by Eyni's Theorem).
After a translation, we may assume that $\phi(e)=0$.
Choose an open identity neighbourhood $O\sub G$
such that $OO\sub U$ and a connected identity neighbourhood $W\sub H$ such that $W\sub\cO$.
Then $S:=O\cap \phi^{-1}(P\times \{0\})$ is a $C^r_\K$-submanifold of~$G$.
If $s=\phi^{-1}(p,0)\in S$, then
$sW$ is an integral manifold for~$D$ contained in $W$ which contains $\phi^{-1}(p,0)$
and hence contained in $S_p:=\phi^{-1}(\{p\}\times Q)$.
As these sets are disjoint for distinct $p$,
we see that $s_1W\cap s_2W=\emptyset$
if $s_1\not=s_2$.
As a consequence, the product map
\[
p\colon S\times W\to U\sub G,\quad (s,w)\mto sw
\]
is injective. It is clear that $p$ is $C^r_\K$.
We have $\phi^{-1}(p,0)H=L_{\phi^{-1}(p,0)}\supseteq S_p$,
for each $x\in U$.
Write $\phi=(\pi,\phi_2)\colon U\to P\times Q$. Then
\[
x=\underbrace{\phi^{-1}(\pi(x),0)}_{=:s(x)}\underbrace{\phi^{-1}(\pi(x),0)^{-1}x}_{=:h(x)}
\]
and both $s\colon U\to U$ and $h\colon U\to H$ are $C^r_\K$.
Then $U_0:=(s,h)^{-1}(O\times W)=(s,h)^{-1}(S\times W)$
is open in $U$ and
\[
p\circ (s,h)=\id_{U_0}
\]
Hence $U_0\sub \im(p)$.
Thus $Y:=p^{-1}(U_0)$ is an open subset of $S\times W$
and $p|_Y\colon Y\to U_0$ is a bijection.
Since $(s,h)\circ p|_Y=\id_Y$,
we see that the $C^r_\K$-map $(s,h)|_{U_0}$ is the inverse of $p|_Y$.
Thus $p|_Y\colon Y\to U_0$ is a $C^r_\K$-diffeomorphism.
There are open identity neighbourhoods $S_1\sub S$ and $W_1\sub W$ such that
$S_1\times W_1\sub Y$. Then $S_1W_1=p(S_1\times W_1)$ is open in $U_0$ and hence in $G$,
and the product map $p|_{S_1\times W_1}\colon S_1\times W_1\to S_1W_1\sub G$
is a $C^r_\K$-diffeomorphism. Therefore Lemma~\ref{lafolk} applies.
\end{proof}
{\bf Proof of Theorem G.}
Let $\ch:=L(H)$ and $D$ be the left invariant distribution on $G$ given by $D_g:=T\lambda_g(\ch)$.
Let $K:=L_e$ be the leaf to $D$ passing through the neutral element~$e$.
We know from Eyni's Theorem that $S$ is an immersed Lie subgroup of~$G$ with $L(K)=\ch$.
Since also the connected component $H_e$ of~$e$ in~$H$ is an integral manifold for~$D$,
the maximality of $S$ shows that $H_0\sub K$
and that the inclusion map $j\colon H_0\to K$ is $C^r_\K$.

The case that $H$ is a Banach-Lie group.
Since $L(j)=\id_\ch$, the inverse function theorem
shows that $j$ is a diffeomorphism onto an open subgroup
and hence and isomorphism of Lie groups (as $K$ is connected).
Hence $K$ is a $C^r_\K$-submanifold of~$G$ and thus $G/K$
admits a $C^r_\K$-manifold structure which turns the canonical
map $p\colon G\to G/K$, $x\mto xK$ into a $C^r_\K$-submersion
(by Proposition~\ref{noregu}).

The case that $H$ is co-Banach and regular.
Since $S$ is a connected $C^r_\K$-manifold,
we find for each $y\in K$ a $C^\infty$-map $\eta\colon [0,1]\to K$ such
$\eta(0)=e$ and $\eta(1)=y$. Let $\gamma:=\delta^\ell(\eta)$ be the left logarithmic derivative of $\eta$.
Since $H$ is regular, we can form $\zeta:=\Evol_H(\gamma)$.
Then both $\eta$ and $\zeta$ satisfy the initial value problem
\[
y'(t)=y(t)\gamma(t),\quad y(0)=e
\]
in~$G$ and hence (by the uniqueness of solutions)
$\eta=\zeta$, entailing that $x=\eta(1)=\zeta(1)\in H_0$.
Hence $K\sub H_0$ and thus $K=H_0$ as a set.
As the inclusion map $K\to G$ is $C^r_\K$ and $H_0$ is a $C^r_\K$-submanifold
of~$G$, we deduce that $\id\colon K\to H_0$ is $C^r_\K$.
As we already discussed the inverse $j$, we see that $\id\colon K\to H_0$
is a $C^r_\K$-diffeomorphism. Hence Proposition~\ref{noregu}
provides a $C^r_\K$-manifold structure on~$G/K$ which turns $p\colon G\to G/K$, $x\mto xK$
into a $C^r_\K$-submersion.

In either case, Lemma~\ref{lafolk}
provides a $C^r_\K$-submanifold $S\sub G$ with $e\in S$
and an open identity neighbourhood $W\sub H_0$
such that $SW$ is open in~$G$ and the product map
\[
S\times W\to SW
\]
is a $C^r_\K$-diffeomorphism.
Since $W$ is also an open identity neighbourhood in~$H$,
we deduce with Lemma~\ref{lafolk}
that there is a $C^r_\K$-manifold structure on $G/H$
which turns $q\colon G\to G/H$, $x\mto xH$ into a $C^r_\K$-submersion.
If $H$ is a normal subgroup, then $G/H$ is a Lie group,
again by Lemma~\ref{lafolk}.\,\Punkt
\section{Proof of Theorem H}
We know that every $C^r_\K$-immersion is a na\"{\i}ve immersion.
Conversely, let $f\colon M\to N$ be a $C^r_\K$-map between $C^r_\K$-manifolds
such that $f$ is a na\"{\i}ve immersion and condition (a) or (b) of the theorem
is satisfied.
Let $E_1$ be the modelling space of~$M$ and $E_2$ be the modelling space of~$N$.
Given $x\in M$, let $\phi\colon U_\phi\to V_\phi\sub E_1$
and $\psi\colon U_\psi\to V_\psi\sub E_2$ be charts form $M$ and $N$, respectively,
such that $x\in U_\phi$ and $f(U_\phi)\sub U_\psi$.
Abbreviate
\[
g:=\psi\circ f\circ \phi^{-1}\colon V_\phi\to V_\psi.
\]
Since $T_xf$ is a linear topological embedding with complemented image,
also $g'(\phi(x))\colon E_1\to E_2$ is a linear topological embedding with
complemented image. We let $F:=\im\, g'(\psi(x))$ and
choose a vector subspace $Y\sub E_2$ such that
$E_2=F\oplus Y$.
After shrinking the open $\psi(x)$-neighbourhood $V_\psi$ and the open $x$-neighbourhood $V_\phi$,
we may assume that
\[
V_\psi=P\times Q
\]
with open subsets $P\sub F$ and $Q\sub Y$.
Let $g=(g_1,g_2)$ be the components of $g$ with respect to
$E_2=F\oplus Y=F\times Y$. Thus $g_1(V_\phi)\sub P$.
Then
\[
g_1'(\psi(x))=g'(\psi(x))|^F\colon E_1\to F
\]
is an isomorphism of topological vector spaces. Since $E_1$ is a Banach space by hypothesis,
also $F$ is a Banach space. Since $r\geq 1$ if $E_1$ is finite-dimensional
and $r\geq 2$ if $E_1$ is infinite-dimensional,
we can use the Inverse Function Theorem for $C^r_\K$-maps between
Banach space. Thus, after shrinking the open $\psi(x)$-neighbourhood $V_\phi$,
we may assume that $g_1(V_\phi)$ is open in $F$ and
\[
g_1\colon V_\phi\to g(V_\phi)\sub F
\]
is a $C^r_\K$-diffeomorphism.
We choose an isomorphism of topological vector spaces $\alpha\colon E_1\to F$ (for example,
the restriction of $g'(\psi(x))$) and set $W:=\alpha(V_\phi)$.
Then
\[
\theta\colon W\times Q\to M_2,\quad
(w,q)\mto \psi^{-1}(g_1(\alpha^{-1}(w)),q)
\]
is a $C^r_\K$-diffeomorphism onto an open subset $O\sub U_\psi$.
Thus
\[
\theta^{-1}\colon O\to W\times Q\sub F\times Y=E_2
\]
is a chart for $M_2$.
We have
\begin{eqnarray*}
(f\circ \phi^{-1})(u)&=& (\psi^{-1}\circ\psi\circ f\circ\phi^{-1})(u)=\psi^{-1}(g(u))=\psi^{-1}(g_1(u),g_2(u))\\
&=&\psi^{-1}(g_1(\alpha^{-1}(\alpha(u))),g_2(u))=
\theta(\alpha(u),g_2(u))
\end{eqnarray*}
for each $u\in U_\phi$
and hence
\[
\theta^{-1}\circ f\circ \phi^{-1}=(\alpha,g_2)\colon U_\phi\to W\times Q\sub F\times Y.
\]
Note that
\[
\beta\colon W\times Y\to W\times Y, \quad \beta(w,y):=(w,y-g_2(\alpha^{-1}(w))
\]
is a $C^r_\K$-diffeomorphism.
Hence also $\beta\circ\theta^{-1}$ is a chart for $M_2$. By the preceding, we have
\[
(\beta\circ \theta^{-1})\circ f\circ\phi^{-1}(u)=\beta(\alpha(u),g_2(u))=(\alpha(u),g_2(u)-g_2(\alpha^{-1}(\alpha(u))))
=(\alpha(u),0).
\]
Hence $(\beta\circ \theta^{-1})\circ f\circ \phi^{-1}=\alpha|_{V_\phi}\colon V_\phi\to F\sub E_2$.
Thus $f$ is a $C^r_\K$-immersion.\,\Punkt
\section{Proof of Theorem I}
We first prove the theorem and then discuss an example.\\[2.3mm]
{\bf Proof of Theorem~I.}
Since $q$ is a submersion, $G_x=q^{-1}(\{G_x\})$ is a smooth submanifold
of~$G$.
Given $g\in G$, abbreviate $\sigma_g:=\sigma(g,.)$.
We have $(b\circ \lambda_g)(h)=b(gh)=\sigma(gh,x)=\sigma(g,.)(b(h))$
for each $h\in G$ and thus $b\circ \lambda_g=\sigma_g\circ b$, entailing that
\begin{equation}\label{equib}
T\lambda_g\circ Tb=T\sigma_g\circ Tb.
\end{equation}
Likewise,
\begin{equation}\label{equitild}
T\Lambda_g\circ T\wt{b}=T\sigma_g\circ T\wt{b}
\end{equation}
with $\Lambda_g\colon G/G_x\to G/G_x$, $hG_x\mto ghG_x$.
Suppose there is $g\in G$ and $0\not=v\in T_{gG_x}G/G_x$ such that
$T\wt{b}(v)=0$. Using (\ref{equitild}) with $g^{-1}$ in place of $g$,
we see that also $T\wt{b}T\Lambda_{g^{-1}} v=T\sigma_{g^{-1}}T\wt{b}v=0$,
where $0\not=T\Lambda_{g^{-1}}v\in T_{G_x}G/G_x$.
We may therefore assume that $g=e$.
To derive a contradiction, we write $v=Tq(w)$ with some $w\in L(G)$.
Since $\wt{b}\circ q=b$, we have $T\wt{b}\circ Tq=Tb$.
Thus $Tb(w)=0$. Since $T_eq=L(G_x)$
by Lemma~\ref{lafolk} and $v=T=eq(w)\not=0$, we have
\begin{equation}\label{thecontra}
w\not\in L(G_x).
\end{equation}
Now consider the one-parameter group
\[
\gamma\colon \R\to G,\quad \gamma(t):=\exp_G(tw).
\]
Then $\gamma'(t)=T\lambda_{\gamma(t)}v$ for each $t\in \R$
and hence
\[
(b\circ \gamma)'(t)=TbT\lambda_{\gamma(t)}v=T\sigma_{\gamma(t)}Tb(v)=0
\]
for each $t\in \R$, entailing that $b\circ\gamma$ is constant and hence
$b(\gamma(t))=x$ for all $t\in \R$.
We can therefore consider $\gamma$ also as a smooth map to $G_x$
and deduce that $v=\gamma'(0)\in T(G_x)$, contradicting~(\ref{thecontra}).
This completes the proof.\,\Punkt
\begin{example}\label{conjuex}
Let $M$ be a compact smooth manifold
and $H$
be a compact Lie group.
We consider
the conjugation action of $H$
(viewed as the group of constant maps) on the Fr\'{e}chet-Lie group
$G:=C^\infty(M,H)$.
Given $\gamma\in H$, its
stabilizer
\[
H_\gamma=\{h\in H\colon h\gamma h^{-1}=\gamma\}
\]
is a closed subgroup of~$H$ and thus $H/{H_\gamma}$ is a compact
real analytic manifold. As the quotient map $q\colon H\to H/H_\gamma$
is a real analytic submersion and the orbit map
\[
H\to G,\quad h\mto h\gamma h^{-1}
\]
is real analytic, also the induced injective map
\[
\wt{b} \colon H/H_\gamma\to G,\quad h H_\gamma\mto h\gamma h^{-1}
\]
is real analytic. Since $H$ is finite-dimensional, it has a real analytic exponential
map and thus Theorem~I
shows that $T\wt{b}$ is injective. Since $H/H_\gamma$ is finite-dimensional,
this means that $\wt{b}$ is a na\"{\i}ve immersion
and hence a real analytic immersion, by Theorem~H.
But $\wt{b}$ is also a topological embedding (since $H/H_\gamma$ is compact).
Hence $\wt{b}$ is a real analytic embedding and thus $\gamma^H=\wt{b}(H/H_\gamma)$
is a real analytic submanifold of $C^\infty(M,H)$ and $\wt{b}\colon H/H_\gamma\to \gamma^H$
is a real analytic diffeomorphism, by Lemma~\ref{easieremb}.
\end{example}
\section{Proof of Theorem J}
For each $x\in M$, we choose a $C^r$-diffeomorphism
$\phi_x \colon U_x \to V_x$ from an open neighbourhood $U_x\sub M$ of~$x$
onto an open subset $V_x\sub E$.
As we assume that $M$ is $C^r$-regular,
we find a $C^r$-function $h_x \colon M \to \R$
with $\Supp(h_x)\sub U_x$
such that $h_x|_{W_x}=1$
for some open neighbourhood $W_x\sub U_x$ of~$x$.
We also find a $C^r$-function $g_x \colon M \to\R$
such that $g_x(x)=1$ and $g_x|_{M\setminus W_x}=0$.
Now
\[
\psi_x\colon M\to E,\quad
y\mto \left\{
\begin{array}{cl}
h_x(y)\phi_x(y) & \mbox{if $y\in U_x$;}\\
0 &\mbox{if $y\in M\setminus\Supp(h_x)$}
\end{array}
\right.
\]
is a well-defined $C^r$-map.
We consider the $C^r$-map
\[
\theta := (\psi_x,g_x)_{x\in M} \colon M \to (E \times \R)^M .
\]
To see that~$\theta$ is injective, let $x,y\in M$ such that
$\theta(x)=\theta(y)$; then $g_x(y)=g_x(x)=1$
and thus $y \in W_x$, whence $h_x(y)=1$ and thus
$\phi_x(x)=\psi_x(x)=\psi_x(y)=\phi_x(y)$, entailing that $x=y$.\\[2.3mm]
To see that the image $\theta(M)$ is a $C^r$-submanifold of $(E \times \R)^M$,
consider a point $\theta(x)\in \theta(M)$ with $x\in M$.
The set
\[
Y_x:=\{y\in M\colon g_x(y)\not=0\}
\]
is an open neighbourhood of~$x$ in~$W_x$.
The set
\[
Z_x := \phi_x(Y_x)
\]
is an open neighbourhood of $\phi_x(x)$ in~$V_x$.
Consider the evaluation
\[
\ev_x \colon  (E \times \R)^M \to E \times \R,\quad f\mto f(x)
\]
at~$x$, which is a continuous linear map.
Abbreviate $\R^\times:=\R\setminus\{0\}$.
Then
\[
Q_x := \ev_x^{-1}(E \times\R^\times)
\]
is an open neighbourhood of $\theta(x)$ in~$(E\times\R)^M$.
We show that
\begin{equation}\label{isgrap}
Q_x \cap \theta(M)
\end{equation}
is the graph of the $C^r$-function
\[
f_x \colon Z_x \to \R \times (E \times \R)^{M \setminus\{x\}}
\]
with first component
$g_x \circ \phi_x^{-1} |_{Z_x}$
and $y$-component
\[
(\psi_y, g_y) \circ \phi_x^{-1} |_{Z_x}
\]
for $y\in M\setminus \{x\}$.
In fact, if $y\in M$ such that
$\theta(y) \in Q_x$, then $g_x(y)\not=0$, 
whence $y \in Y_x$ and thus $y = \phi_x^{-1}(z)$ for some
$z \in Z_x$, entailing that
\[
\theta(y)=\theta(\phi_x^{-1}(z)) = (z,f_x(z)).
\]
Conversely, $(z,f_x(z))=\theta(\phi_x^{-1}(z))\in Q_x\cap \theta(M)$
for each $z\in Z_x$. Thus (\ref{isgrap}) holds.\\[2.3mm]
As graphs of $C^r$-maps are $C^r$-submanifolds
and $\theta(M)$ locally looks like such within $(E\times\R)^M$,
we see that
$\theta(M)$ is a $C^r$-submanifold in $(E \times \R)^M$.
Since $\theta$ is $C^r$ as a map to $(E\times\R)^M$
and $\theta(M)$ (being locally a graph) is modelled
on split (and hence closed) vector subspaces of $(E\times \R)^M$,
we deduce that $\theta$ is $C^r$ also as a map to the
$C^r$-submanifold $\theta(M)$.
Moreover,
$\theta^{-1} \colon  \theta(M) \to M$ is $C^r$
as it is given as the projection onto the first
component on the above graph around $\phi(x)$, composed
with $\phi_x^{-1}$.\Punkt
\appendix
\section{Proof of Lemma~\ref{lafolk}}\label{apphomog}
(a)$\impl$(b):
If $q$ is a $C^r_\K$-submersion, then Lemma~\ref{propsec}
provides a $C^r_\K$-section $\sigma\colon P\to G$
for $q$, defined on an open neighbourhood $P$ of $q(e)$ in~$G/H$.
Thus $q\circ\sigma=\id_P$.
Define $S:=\sigma(P)$.
Then
\[
SH=q^{-1}(P)
\]
is an open subset of~$G$.
The map
\[
\psi\colon P\times H\to SH,\quad (x,h)\to\sigma(x)h
\]
is $C^r_\K$ and a bijection.
The inverse map is
\[
SH\to P\times H,\quad z\mto (q(z),\sigma(q(x))^{-1}z)
\]
and hence $C^r_\K$.
Thus $\psi$ is a $C^r_\K$-diffeomorphism.
Since $P\times\{e\}$ is a $C^r_\K$-submanifold of $P\times H$,
we deduce that $S=\psi(P\times\{e\})$
is a $C^r_\K$-submanifold of the open subset $SH$ of~$G$ and hence of~$G$.
It remains to note that the product map
\[
p\colon S\times H\to SH
\]
is a $C^r_K$-diffeomorphism because $h\colon P\times H\to S\times H$,
$h(x,h):=(\sigma(x),h)=(\psi(x,e),h)$ is a $C^r_\K$-diffeomorphism
and $p\circ h=\psi$ is a $C^r_\K$-diffeomorphism.

(b)$\impl$(c) is trivial, as we can choose $W:=H$.

(c)$\impl$(b): Let $S\sub G$ be a $C^r_\K$-submanifold with $e\in S$
and $W\sub H$ be an open identity neighbourhood such that
$SW$ is open in $G$ and the product map $p\colon S\times W\to SW$
is a $C^r_\K$-diffeomorphism.
Let $V\sub G$ be an open identity neighbourhood such that $W=H\cap V$
and $S_1\sub S$ b an open identity neighbourhood such that
$(S_1)^{-1}S_1\sub V$ and hence
\[
H\cap S_1^{-1}S_1\sub W.
\]
Then $p(S_1\times W)=S_1W$ is open in~$G$, entailing that also the set
$S_1H=S_1WH=\bigcup_{h\in H}S_1Wh$
is open in~$G$.
If $s_1,s_2\in S_1$ and $h_1,h_2\in H$ are such that
$s_1h_1=s_2h_2$, then $s_2^{-1}s_1=h_2h_1^{-1}\in H\cap S_1^{-1}S_1\sub W$
and hence $p(s_1,e)=s_1=s_2(h_2h_1^{-1})=p(s_2,h_2h_1^{-1})$
implies that $s_1=s_2$ and $h_1=h_2$.
As a consequence,
the product map
\[
p_1\colon S_1\times H\to S_1H
\]
is a bijection. We have already seen that the image is open.
Moreover, $p_1$ is $C^r_\K$ and equivariant for the
right $C^r_\K$-actions of $H$ on $S_1\times H$ and $G_1H$ given by $(s,h).h_1:=(s,hh_1)$ and
$g.h:=gh$, respectively.
The map $p_1$ induces a $C^r_\K$-diffeomorphism from the open set $S_1\times W$ onto its open image.
For $h_1\in H$, we have $p_1(s,h)=p_1(s,hh_1^{-1})h_1$ for all $(s,h)$
in the open set $(S_1\times W)h_1$.
As the open sets $(S_1\times W)h_1$ cover $S_1\times H$ for $h_1\in H$,
we deduce that $p_1$ is a $C^r_\K$-diffeomorphism.

(b)$\impl$(a): Let $S\sub G$ be a $C^r_\K$-submanifold with $e\in S$ such that
$SH$ is open in $G$ and the product map $p\colon S\times H\to SH$
is a $C^r_\K$-diffeomorphism.
We endow $G/H$ with the quotient topology which makes $G/H$ a Hausdorff space
(as $H$ is closed in $G$) and turns $q\colon G\to G/H$ into a continuous and open,
surjective map. Consider the map
\[
m\colon G\times G/H\to G/H,\quad (g,xH)\mto (gx)H
\]
and the group multiplication $\mu\colon G\times G\to G$.
Since $\id_G\times q\colon G\times G\to G\times G/H$
is a continuous and open surjection and hence a quotient map,
we deduce from the continuity of $m\circ  (\id_G\times q)=q\circ m$
that $m$ is continuous.
In particular,
\[
m_g:=m(g,.)\colon G/H\to G/H
\]
is a homeomorphism (with inverse $m_{g^{-1}}$),
for each $g\in G$.
Since $SH$ is open in~$G$ and $q$ is an open map,
we deduce that
\[
U:=q(SH)=q(S)
\]
is open in~$G/H$.
The map $q|_S\colon S\to U$ is bijective (since $p$ is a bijection)
and open, because $q|_S(V)=q(V\times H)$
is open in~$U$ for each open subset $V\sub S$.
For each $g\in G$, $gU:=m_g(U)$ is open in $G/H$.
We define
\[
\psi_g:=m_g\circ q|_S \colon S\to gU.
\]
Because $m_g$ is a homeomorphism, $gU$ is open in $G/H$.
Moreover, $\psi_g$ is a homeomorphism as it is a composition of homeomorphisms.
If $g_1,g_2\in G$, then $g_1U\cap g_2U$ is open in~$G/H$
and
\begin{eqnarray*}
\psi_{g_1}^{-1}(g_1U\cap g_2U)&=& \{
x\in S\colon (\exists y\in S)\; g_1xH=g_2yH\}\\
&=&\{x\in S\colon g_2^{-1}g_1x\in SH\} 
=
S\cap g_1^{-1}g_2SH.
\end{eqnarray*}
Moreover,
$\psi_{g_2}^{-1}\circ \psi_{g_1}$ is the $C^r_\K$-map
\[
S\cap g_1^{-1}g_2SH\to S\cap g_2^{-1}g_1SH,\quad x\mto g_2^{-1}g_1x.
\]
Therefore the maps $\phi\circ \psi_g^{-1}|_{\psi_g(U_\phi)}$,
with $\phi\colon U_\phi\to V_\phi$ ranging through the charts of~$S$
is a $C^r_\K$-atlas for $G/H$ and is contained in a maximal
$C^r_\K$-atlas which makes $G/H$ a $C^r_\K$-manifold and turns each of the
maps $\psi_g$ into a $C^r_\K$-diffeomorphism.
Now $\theta\colon gU\times H\to gSH$, $\theta(x,h):=g\psi_g^{-1}(x)h$
is a $C^r_\K$-diffeomorphism for $g\in G$ such that
$gSH$ is an open neighbourhood of~$g$ in~$G$
and
\[
(q\circ \theta)(x,h)=q(\lambda_g(\psi_g^{-1}(x))h)
=q(\lambda_g\psi_g^{-1}(x))=q(q|_{gS}^{-1}(x))=x
\]
using that $\psi_g=m_g\circ q=q\circ \lambda_g$,
hence $\psi_g\circ \lambda_{g^{-1}}=q$ and hence
$\lambda_g\circ \psi_g^{-1}=(q|_{gS})^{-1}$.
Thus $q$ is a $C^r_\K$-submersion,
by Lemma~\ref{locapro}.\\[2.3mm]
To prove the remaining assertions,
let $m$ be as before and $\mu\colon G\times G\to G$ be the group multiplication.
Then $\id_G\times q\colon G\times G\to G\times G/H$ is a surjective $C^r_\K$-submersion
and $m\circ (\id_G\times q)=\mu$ is a $C^r_\K$-map.
The action $m$ is therefore $C^r_\K$, by Lemma~\ref{checkdiff}.
Since $q(H)=\{q(e)\}$, we have $\ch:=T_eH\sub \ker T_eq$.
Because $p$ in~(b) is a diffeomorphism
and $q|_S\colon S\to q(S)$ a diffeomorphism,
for each $v\in T_eG =T_eS\oplus \ch$
which is not in $\ch$, we have $v=s+h$ with $0\not=s\in T_eS$ and $h\in \ch$.
As $T_{(e,e)}p$ is the addition map $T_eS\times \ch\to T_eG$,
we deduce that $T_eq(v)=T_eq(T_ep(.,e)(s)+h)=T_e(q\circ p(.,e))(s)\not=0$
as $q\circ p(.,e)=\psi_e$
is a diffeomorphism.

If, finally, $N$ is a normal subgroup, write $\bar{\mu}$ for the group multiplication
of $G/H$. Since $q\times q$ is a $C^r_\K$-submersion and $\bar{\mu}\circ (q\times q)=q\circ\mu$
is $C^r_\K$, we deduce with Lemma~\ref{checkdiff} that $\bar{\mu}$ is $C^r_\K$.
Likewise, $\bar{\eta}=q\circ\eta$ holds for the group inversion maps,
and thus $\bar{\theta}$ is $C^r_\K$.\,\Punkt
{\small
Helge  Gl\"{o}ckner, Universit\"at Paderborn,
Institut f\"{u}r Mathematik,\\
Warburger Str.\ 100, 33098 Paderborn, Germany;
\,Email: {\tt glockner@math.upb.de}\vfill
\end{document}